\documentclass[a4paper,11pt]{article} 
\usepackage{
amsmath,
amsthm,
amsfonts,
amssymb,
enumerate,
xfrac,
tikz,
comment,
pgfplots,
xcolor,
mathrsfs
}
\pgfplotsset{compat=newest,compat/show suggested version=false}
\usepackage[colorlinks,citecolor=blue,urlcolor=blue]{hyperref}
\usetikzlibrary{shapes.geometric,positioning,matrix,intersections,calc,patterns,arrows.meta}
\excludecomment{codes}
\usepackage[text={6.5in,9.0in},centering]{geometry}

\usepackage{graphicx,url,subfig,diagbox}
\usepackage{authblk,makecell}
\RequirePackage[OT1]{fontenc}
\usepackage{datetime}
\usepackage[normalem]{ulem} 
\usepackage{soul} 
\numberwithin{equation}{section}
\numberwithin{figure}{section}

\usepackage{mathtools}
\usepackage{amssymb}
\usepackage{pifont}
\usepackage{comment}
\usepackage{lmodern}

\usepackage{amsrefs}

\renewcommand{\MR}[1]{~\href{https://mathscinet.ams.org/mathscinet-getitem?mr=MR#1}{MR#1}.}

\BibSpec{article}{%
    +{}  {\PrintAuthors}                {author}
    +{,} { \textit}                     {title}
    +{.} { }                            {part}
    +{:} { \textit}                     {subtitle}
    +{,} { \PrintContributions}         {contribution}
    +{.} { \PrintPartials}              {partial}
    +{,} { }                            {journal}
    +{}  { \textbf}                     {volume}
    +{}  { \PrintDatePV}                {date}
    +{,} { \issuetext}                  {number}
    +{,} { \eprintpages}                {pages}
    +{,} { }                            {status}
    +{,} { available at \eprint}        {eprint}
    +{}  { \parenthesize}               {language}
    +{}  { \PrintTranslation}           {translation}
    +{;} { \PrintReprint}               {reprint}
    +{.} { }                            {note}
    +{.} {}                             {transition}
    +{}  {\SentenceSpace \PrintReviews} {review}
}

\theoremstyle{plain}
\newtheorem{theorem}{Theorem}[section]
\newtheorem{corollary}[theorem]{Corollary}
\newtheorem{lemma}[theorem]{Lemma}
\newtheorem{proposition}[theorem]{Proposition}

\theoremstyle{definition}

\newtheorem{remark}[theorem]{Remark}

\newcommand{\E}{\mathbb{E}}

\newcommand{\ud}{\ensuremath{\mathrm{d} }}

\newcommand{\Norm}[1]{\left\| #1 \right\|}

\newcommand{\ip}[2]{\left\langle #1,#2\right\rangle}

\DeclareMathOperator{\Cov}{Cov}

\newcommand{\calG}{\mathcal{G}}

\newcommand{\R}{\mathbb{R}}
\renewcommand{\P}{\mathbb{P}}

\newcommand{\Erfc}{\ensuremath{\mathrm{erfc} }}
\newcommand{\ML}[2]{E_{#1,\,#2}}  

\newcounter{Pq}
\newcounter{Le}
\usepackage[framemethod=tikz]{mdframed}
\mdfsetup{%
middlelinecolor=lightgray,
middlelinewidth=2pt,
backgroundcolor=red!20,
roundcorner=10pt}
{\refstepcounter{Pq}
\begin{mdframed}[backgroundcolor=lightgray]\textbf{Qanqiu's comment \thePq \: #1}\\[+1.5em]}%
{\end{mdframed}}
{\refstepcounter{Le}
\begin{mdframed}[backgroundcolor=red!10]\textbf{Le's comment \theLe \: #1}\\[0.5em]}%
{\end{mdframed}}
\numberwithin{Pq}{section}
\numberwithin{Le}{section}

\title{
Spatial covariance of KPZ from flat initial profile
}

\author[1]{Le Chen\thanks{Email: \href{mailto:lzc0090@auburn.edu}{lzc0090@auburn.edu}}}
\author[1]{Juan J. Jim\'enez\thanks{Email: \href{mailto:juj0003@auburn.edu}{juj0003@auburn.edu}}}
\affil[1]{Department of Mathematics and Statistics\protect\\
Auburn University, Auburn, AL 36849, USA}

\date{}

\begin{document}
\maketitle \setlength{\parindent}{1.5em}

\begin{abstract}
	We study the fixed-time spatial covariance of the KPZ equation with flat
	initial profile. Using Malliavin calculus and a Clark--Ocone representation,
	we show that as $|x|\to\infty$, $\Cov[h(t,x),h(t,0)]$ is governed by a
	boundary-layer regime near the initial time and satisfies
	\[
		\Cov[h(t,x),h(t,0)] \sim \kappa(t) \int_0^t p_{2r}(x)\,\ud r = \frac{2\kappa(t)}{\sqrt{\pi}}\,t^{3/2}\,|x|^{-2}\exp\!\left(-\frac{x^2}{4t}\right),
		\quad |x|\to\infty,
	\]
	where $\kappa(t) = \big(\E[Z(t,0)^{-1}]\big)^2$, $Z$ is the flat
	stochastic heat equation solution, and $p_t$ is the one-dimensional heat
	kernel. In sharp contrast with the narrow-wedge regime, where Gu--Pu
	\cite{gu.pu:25:spatial}*{Theorem~1.1} proved that for each fixed $t>0$,
	\[
		\Cov\!\left[h^{\mathrm{nw}}(t,x),h^{\mathrm{nw}}(t,0)\right]\sim \frac{t}{|x|},
		\qquad |x|\to\infty,
	\]
	the flat initial profile exhibits Gaussian decay, yielding, to the best of our knowledge, the first exact
	spatial covariance asymptotic for the KPZ equation under flat initial data.
	We also establish an explicit closed-form formula for the second moment of
	the continuum directed random polymer partition function.
\end{abstract}


	{\hypersetup{linkcolor=black}
		\tableofcontents
	}

\section{Introduction}\label{S:introduction}

In this paper, we study the Kardar--Parisi--Zhang (KPZ) equation~\cite{kardar.parisi.ea:86:dynamic}
\eqref{E:kpz}:
\begin{equation}\label{E:kpz}
	\partial_t h(t,x)=\tfrac12\partial_x^2 h(t,x)+\tfrac12\big(\partial_x h(t,x)\big)^2+\xi(t,x),
	\qquad t>0,\ x\in\R,
\end{equation}
subject to \textit{zero initial data}, $h(0,x) = 0$. Its Cole--Hopf solution is
defined by $h(t,x) = \log Z(t,x)$, where $Z$ solves the stochastic heat equation
(SHE),
\begin{equation}\label{E:she}
	\partial_t Z(t,x)=\tfrac12\partial_x^2 Z(t,x)+ Z(t,x)\,\xi(t,x),
	\qquad t>0,\ x\in\R,
\end{equation}
subject to the \textit{flat initial condition}:
\begin{equation}\label{E:flat-initial}
	Z(0,x) = u_0(x) \coloneqq 1.
\end{equation}
In both~\eqref{E:kpz} and~\eqref{E:she}, $\xi$ is a space--time white noise on
$\R_+\times\R$ on some probability space $(\Omega,\mathcal{F},\P)$, and
$(\mathcal{F}_t)_{t\ge 0}$ denotes the complete filtration generated by $\xi$.
The KPZ equation (and related polymer models) has been studied extensively; we
refer to \cite{corwin:12:kardar-parisi-zhang} for a broad overview. Our focus is
the \emph{fixed-time} spatial covariance of the height field (under flat initial
data, the law is stationary in $x$). In the narrow-wedge setting, Gu--Pu
\cite{gu.pu:25:spatial} determine the large-$|x|$ asymptotic of
$\Cov[h(t,x),h(t,0)]$ and show that the dominant contribution comes from a
boundary-layer geometry near the initial time. Under flat initial data,
integrability and upper bounds for spatial covariances have been obtained via
Malliavin calculus and functional inequalities (see, e.g.,
\cite{chen.khoshnevisan.ea:23:central} and the discussion in
\cite{gu.pu:25:spatial}).

Among the six fundamental KPZ initial geometries highlighted by Corwin
\cite[Section~2.2 and Figure~2]{corwin:12:kardar-parisi-zhang}, the flat
profile is one of the most canonical: together with narrow-wedge and Brownian
data, it gives a basic diffusive fixed point leading respectively to Airy$_1$,
Airy$_2$, and stationary limits, and to GOE, GUE, and Baik--Rains
fluctuations (see \cite{corwin:12:kardar-parisi-zhang} and references therein). Historically, however, exact results for flat data have lagged
behind those for narrow wedge, reflecting the greater difficulty of the flat
regime: the Airy$_2$ covariance decay was first observed by
Pr\"ahofer--Spohn~\cite{prahofer.spohn:02:scale} and rigorously proved by
Widom~\cite{widom:04:on} (see also
\cite{shinault.tracy:11:asymptotics} for higher-order terms), whereas the
super-exponential Airy$_1$ decay was only identified much later in
\cite{basu.busani.ea:23:on}.
The present paper continues
this timeline at the level of fixed-time KPZ covariances: Gu--Pu resolved the
narrow-wedge case, and we now resolve the flat case.

Motivated by \cite{gu.pu:25:spatial}, we study how the fixed-time, large-$|x|$
covariance depends on the initial profile \eqref{E:flat-initial}. The main
result of this paper is Theorem~\ref{thm:flat-main} below, which identifies the
exact asymptotic of the covariance $\Cov[h(t,x),h(t,0)]$ for flat initial data.

\begin{theorem}[Flat covariance asymptotic]\label{thm:flat-main}
	For flat initial data \eqref{E:flat-initial}, for all $t>0$, it holds that
	\begin{equation}\label{E:flat-goal}
		\begin{aligned}
			\Cov[h(t,x),h(t,0)]
			                             & \sim \kappa(t)\int_0^t p_{2r}(x)\,\ud r
			= \frac{2\kappa(t)}{\sqrt{\pi}}\,t^{3/2}\,
			|x|^{-2}\exp\!\left(-\frac{x^2}{4t}\right),
			\qquad |x|\to\infty,                                                              \\
			\text{where}\quad  \kappa(t) & \coloneqq \Big(\E\!\left[Z(t,0)^{-1}\right]\Big)^2
			\quad \text{and} \quad p_t(x) = \frac{1}{\sqrt{2 \pi t}} \exp\left(- \frac{x^2}{2t}\right).
		\end{aligned}
	\end{equation}
\end{theorem}

To our knowledge, Theorem~\ref{thm:flat-main} provides the first exact spatial
covariance asymptotic for the KPZ equation under flat initial data.

We refer to the asymptotic relation \eqref{E:flat-goal} as the flat-profile
prediction. It is natural to view \eqref{E:flat-goal} together with the
covariance decay of the universal KPZ scaling limits: the flat and narrow-wedge
initial conditions lead (after the $1\!:\!2\!:\!3$ KPZ scaling) to different
Airy processes and, correspondingly, very different spatial decorrelation
behaviors. We refer to \cite{gu.pu:25:spatial} for a detailed discussion and
further references. See also \cite{widom:04:on,shinault.tracy:11:asymptotics}
for rigorous Airy$_2$ covariance asymptotics and
\cite{basu.busani.ea:23:on} for correlation decay of the Airy$_1$ process.

At the level of the KPZ scaling limits, the contrast between droplet/Airy$_2$
and flat/Airy$_1$ is particularly sharp. For positive functions $f$ and $g$, we
write $f(x)\sim g(x)$ as $x\to\infty$ if $f(x)/g(x)\to 1$ (in particular, $\sim$
identifies the leading constant), while $f(x)\asymp g(x)$ means two-sided bounds
up to multiplicative constants for all sufficiently large $x$. With this
convention,
\begin{align}\label{E:nw-flat-Airy}
	\begin{aligned}
		\Cov[\mathcal{A}_2(x),\mathcal{A}_2(0)] & = 2x^{-2} + O(x^{-4}),                                  &  & x\to\infty &  & \text{\cite{widom:04:on}},          \\
		\Cov[\mathcal{A}_1(x),\mathcal{A}_1(0)] & = \exp\!\left(-\left(\frac{4}{3}+o(1)\right)x^3\right), &  & x\to\infty &  & \text{\cite{basu.busani.ea:23:on}}.
	\end{aligned}
\end{align}
The polynomial decay in the Airy$_2$ case and the stretched-exponential decay in
the Airy$_1$ case have a natural geometric interpretation in polymer/LPP models:
for droplet data the two profiles are influenced by space--time regions with a
macroscopic overlap (curvature forces geodesics/polymers to coalesce near a
pinned endpoint), whereas in the flat case typical influential regions become
essentially disjoint at large separation and correlations are driven by rare
coalescence events.

The fixed-time, large-separation covariances of the KPZ equation exhibit an
analogous dichotomy. Writing $h^{\mathrm{nw}}$ for the KPZ solution started from
narrow wedge and $h^{\mathrm{flat}}$ for the flat solution, Gu--Pu
\cite{gu.pu:25:spatial}*{Theorem~1.1} and \eqref{E:flat-goal} together yield
that for each fixed $t>0$,
\begin{equation}\label{E:nw-flat-KPZ}
	\begin{aligned}
		\Cov\!\left[h^{\mathrm{nw}}(t,x),h^{\mathrm{nw}}(t,0)\right]     & \sim \frac{t}{|x|},                      &  & |x|\to\infty, \\
		\Cov\!\left[h^{\mathrm{flat}}(t,x),h^{\mathrm{flat}}(t,0)\right] & \sim \kappa(t)\int_0^t p_{2r}(x)\,\ud r, &  & |x|\to\infty.
	\end{aligned}
\end{equation}
Thus the covariance decay changes from polynomial (narrow wedge) to Gaussian
(flat), paralleling the Airy$_2$ versus Airy$_1$ contrast above.
Heuristically, delta initial data pins every directed polymer at the origin,
so $h^{\mathrm{nw}}(t,x)$ and $h^{\mathrm{nw}}(t,0)$ share noise near time
zero regardless of separation, producing the slow $|x|^{-1}$ tail; flat
initial data lets the dominant paths originate from well-separated regions,
and their noise overlap is mediated by heat-kernel diffusion, imprinting the
Gaussian factor $e^{-x^{2}/(4t)}$.

\begin{remark}[Strategy and main difficulties]\label{R:strategy}
	The proof follows the Gu--Pu blueprint, with two additional flat-profile
	inputs. First, we use the Clark--Ocone formula to represent the covariance by
	an explicit deterministic kernel multiplied by a random boundary-layer factor.
	Second, we show a boundary-layer reduction: only times $s\lesssim
		|x|^{-2+\varepsilon}$ contribute at leading order as $|x|\to\infty$. Third, on
	this boundary layer we pass to a ``future-only'' normalization and then use
	shear invariance and shear mixing to identify the limiting constant, which
	reduces to the one-point inverse moment $\E[Z(t,0)^{-1}]$. The technical
	bottlenecks are uniform moment/inverse-moment bounds needed to justify the
	normalizations and the required uniformity in $(s,y)$ on the boundary-layer
	window $|y-x/2|=O(1)$.
\end{remark}

This study crucially relies on the properties of the Green's function
corresponding to
\begin{equation}\label{E:G}
	\begin{dcases}
		\partial_t \calG_{\beta}(t,x;s,y)
		= \tfrac12 \partial_x^2 \calG_{\beta}(t,x;s,y)
		+ \beta\, \calG_{\beta}(t,x;s,y)\, \xi(t,x),
		\qquad t > s \ge 0,\; x,y \in \R, \\
		\calG_{\beta}(s,\cdot\,;s,y) = \delta_y(\cdot),
	\end{dcases}
\end{equation}
and its normalized version $\bar{\calG}_\beta$ defined by
\begin{equation}\label{E:barG}
	\bar{\mathcal{G}}_\beta(t,x;s,y)
	\coloneqq \frac{\mathcal{G}_\beta(t,x;s,y)}{p_{t-s}(x-y)}, \quad \text{for $t \ge s$ and  $x, y\in \R$.}
\end{equation}
Here, $\beta \in \R$. For simplicity of notation, throughout the paper, we drop
the superscript $\beta$ when $\beta = 1$, that is, $\mathcal{G} =
	\mathcal{G}_1$. It is easy to see that $\bar{\calG}_\beta$ satisfies the
following integral equation
\begin{align}\label{E:barG-mild}
	\bar{\calG}_\beta(t,x;s,y) = 1 + \beta \int_s^t \int_{\R} q_{r-s}^{x-y,t-s}(z-y)\, \bar{\calG}_\beta(r,z;s,y) \, \xi(\ud r, \ud z),
\end{align}
where $q_u^{x,\tau}(z)$ is the Brownian-bridge density
\begin{equation}\label{E:bridge-kernel-Rd}
	\begin{aligned}
		q_u^{x,t}(z)
		 & = \frac{p_{t-u}(x-z)\,p_u(z)}{p_t(x)}
		= p_{\frac{u(t-u)}{t}} \left(z-\frac{u}{t}x\right)                                                                \\
		 & = \left(\frac{t}{2\pi u(t-u)}\right)^{1/2} \exp\!\left(-\frac{t}{2u(t-u)}\left|z-\frac{u}{t}x\right|^2\right),
		\quad \text{for $u\in(0,t)$, $x, z \in \R$.}
	\end{aligned}
\end{equation}

The Green's function $\calG$ corresponds to the \emph{point-to-point partition
	function} of the \textit{continuum directed random polymer} (CDRP). When $d =
	1$, in~\cite{alberts.khanin.ea:14:continuum, alberts.janjigian.ea:22:greens},
the Green's function $\calG_\beta$ and its normalized version
$\bar{\calG}_\beta$ have been studied in detail. The higher-dimensional
counterpart has also been recently studied in~\cite{chen.ouyang.ea:26:class}. The
normalized Green's function $\bar{\calG}_\beta$ is a fundamental object in the
theory of the stochastic heat equation, with connections to the Brownian bridge,
the continuum directed random polymer, and the Malliavin derivative of the
solution. We will explain these connections and related work in more detail in
Section~\ref{S:moments} below. We also derive an
explicit second-moment formula for $\bar{\mathcal{G}}_\beta$. We use $\Phi(x)$
to denote the standard normal cumulative distribution function, namely,
\begin{equation}\label{E:Phi}
	\Phi(x) \coloneqq \frac{1}{\sqrt{2\pi}}\int_{-\infty}^{x}e^{-t^2/2}\,\ud t.
\end{equation}

\begin{theorem}[Explicit second moment of $\bar{\mathcal{G}}_\beta$]\label{thm:barG-second-moment-intro}
	For all $t\ge s$, and $x, y \in\R$ with $\tau \coloneqq t-s$, it holds that
	\begin{equation}\label{E:2nd-Mom-1d-intro}
		\E\left[\bar{\calG}_\beta(t,x;s,y)^2\right]
		=1+\sqrt{\pi\beta^4\tau}\,e^{\beta^4\tau/4}\,
		\Phi\!\left(\beta^2\sqrt{\frac{\tau}{2}}\right),
	\end{equation}
	which is the unique solution of the Volterra integral equation
	\begin{equation}\label{E:volterra-1d-intro}
		m_2(\tau)=1+\frac{\beta^2\sqrt{\tau}}{2\sqrt{\pi}}
		\int_0^\tau\frac{m_2(u)}{\sqrt{u(\tau-u)}}\,\ud u,
		\qquad m_2(0)=1.
	\end{equation}
	In particular, $\E\left[\bar{\calG}_\beta(t,x;s,y)^2\right]$ does not depend
	on $x$ and $y$ and only depends on $t-s\ge 0$.
\end{theorem}

\paragraph{Organization of the paper.} Section~\ref{S:clark-ocone} records the
Clark--Ocone covariance representation and basic moment bounds, including a
Gaussian-tail upper bound and a boundary-layer simplification.
Section~\ref{S:moments} establishes moment and inverse-moment bounds for the
normalized Green's function, including an explicit closed-form second moment
(Theorem~\ref{thm:barG-second-moment-intro}). Section~\ref{SS:shear-mixing}
proves that the shear automorphisms are strongly mixing on the white-noise
probability space (Lemma~\ref{L:shear-mixing}), which is the key ingredient for
identifying the boundary-layer constant in Section~\ref{S:spatial-decay}.
Section~\ref{S:spatial-decay} proves Theorem~\ref{thm:flat-main} by (i) reducing
the large-$|x|$ covariance asymptotics to a boundary-layer constant and (ii)
identifying this constant via a future-only profile and shear mixing.

\section{Clark--Ocone representation and basic bounds}\label{S:clark-ocone}

\subsection{Clark--Ocone covariance representation}\label{SS:clark-ocone-cov}

Let $D_{s,y}$ denote the Malliavin derivative at $(s,y)\in\R_+\times\R$.  For
SHE solutions with sufficiently integrable deterministic initial data, one has
the identity \eqref{E:dz-green} (see, e.g., standard references on Malliavin
calculus for SPDEs)
\begin{equation}\label{E:dz-green}
	D_{s,y} Z(t,x)=\mathcal{G}(t,x;s,y)\, Z(s,y),\qquad 0<s<t.
\end{equation}
where $\mathcal{G}(t,x;s,y)$ is given in~\eqref{E:G}. The key point (used
extensively by Gu--Pu) is that $\mathcal{G}(t,x;s,y)$ depends only on the noise
in the time interval $(s,t]$ and is independent of $\mathcal{F}_s$, while
$\bar{\mathcal{G}}$ has bounded (both positive and negative) moments on compact
time intervals; see Lemma~\ref{L:barG-moments} below. \medskip

Assuming $h(t,x)\in\mathbb{D}^{1,2}$, the Clark--Ocone formula yields
\begin{equation}\label{E:clark-ocone}
	h(t,x)-\E[h(t,x)]
	=\int_0^t\int_{\R} \rho_{t,x}(s,y)\,\xi(\ud s\,\ud y),
\end{equation}
where the predictable integrand is
\begin{equation}\label{E:rho-def}
	\rho_{t,x}(s,y)
	\coloneqq \E\!\left[D_{s,y} h(t,x)\mid \mathcal{F}_s\right]
	=\E\!\left[\frac{\mathcal{G}(t,x;s,y)\,Z(s,y)}{Z(t,x)}\ \Big|\ \mathcal{F}_s\right].
\end{equation}
Consequently, for $x\in\R$,
\begin{equation}\label{E:cov-basic}
	\Cov\big[h(t,x),h(t,0)\big]
	=\int_0^t\int_{\R} \E\!\left[\rho_{t,x}(s,y)\rho_{t,0}(s,y)\right]\ud y\,\ud s.
\end{equation}
Identity \eqref{E:cov-basic} is the universal starting point: all dependence on
initial data is hidden in $Z(s,\cdot)$ and $Z(t,\cdot)$ inside
\eqref{E:rho-def}.
We will refer to \eqref{E:clark-ocone} as the Clark--Ocone representation of
$h(t,x)$.

For any $(t,x) \in \R_+ \times \R$, \eqref{E:cov-basic} can be rewritten as
\begin{equation}\label{E:cov-KX}
	\Cov\big[h(t,x),h(t,0)\big]
	=\int_0^t\int_{\R} K(t,s,x,y) \E\!\left[ \mathcal{X}(t,s,x,y) \right]\ud y\,\ud s,
\end{equation}
where, for all $ 0 \le s < t$ and  $x, y \in \R$,  we use the following notation:
\begin{gather}\label{E:K}
	K(t,s,x,y) = p_{t-s}(x-y) p_{t-s}(y),\\
	\mathcal{X}(t,s,x,y) = Z(s,y)^2\,\mathcal{B}(t,x;s,y)\,\mathcal{B}(t,0;s,y), \label{E:Xflat-Bflat}\\
	\mathcal{B}(t,x;s,y)
	\coloneqq \E_{s,t}\!\left[\frac{\bar{\mathcal{G}}(t,x;s,y)}{Z(t,x)}\right]. \label{E:Bflat-def}
\end{gather}
Here $\E_{s,t}$ denotes expectation over the noise in the time interval $(s,t]$ only.
We will refer to \eqref{E:cov-KX} as the kernelized covariance representation. Also, we denote by $\Norm{\cdot}_k$ the norm in $L^k(\Omega)$.

\subsection{Known integrability and a Gaussian-tail upper bound}\label{SS:gaussian-tail}

The works \cite{chen.khoshnevisan.ea:22:central,
	chen.khoshnevisan.ea:23:central} study spatial averages under flat initial data.
In particular, for the parabolic Anderson model ($d=1$, space--time white noise,
$\sigma(u)=u$) and $g(u)=\log u$, \cite{chen.khoshnevisan.ea:23:central} proves
that for every $t>0$,
\[
	\int_{\R}\Cov[h(t,x),h(t,0)]\,\ud x\in(0,\infty),
\]
which in particular implies integrability of the covariance. (This
exponential-type spatial decorrelation under flat initial data is also noted in
the introduction of Gu--Pu \cite{gu.pu:25:spatial}.) We collect the moment
estimates used in these bounds in Section~\ref{S:moments}.

For $t \ge0,$ and $x \in \mathbb{R}$, we denote $I_{\mathrm{flat}}(t,x)$ the
proxy integral,
\begin{equation}\label{E:Iflat-def}
	I_{\mathrm{flat}}(t,x)
	\coloneqq \int_0^t\int_{\R} K(t,s,x,y)\,\ud y\,\ud s
	=\int_0^t p_{2(t-s)}(x)\,\ud s
	=\int_0^t p_{2r}(x)\,\ud r.
\end{equation}
By the same boundary asymptotic,
\begin{equation}\label{E:Iflat-asymp}
	I_{\mathrm{flat}}(t,x)
	\sim \frac{2}{\sqrt{\pi}}\,t^{3/2}\,|x|^{-2}\exp\!\left(-\frac{x^2}{4t}\right),
	\qquad |x|\to\infty.
\end{equation}
We will use \eqref{E:Iflat-def} and its boundary asymptotic
\eqref{E:Iflat-asymp} throughout.

\begin{proposition}[Flat initial: Gaussian-tail covariance upper bound]\label{P:flat-cov-upper}
	Fix $t>0$. For the parabolic Anderson model in $1+1$ dimensions with flat initial
	data, there exists $C_t<\infty$ such that for all $x\in\R$,
	\begin{equation}\label{E:flat-cov-upper}
		\big|\Cov[h(t,x),h(t,0)]\big|
		\le C_t \int_0^t p_{2r}(x)\,\ud r
		= C_t\,I_{\mathrm{flat}}(t,x).
	\end{equation}
	In particular,
	$
		\Cov[h(t,x),h(t,0)]
		=O\big(|x|^{-2}e^{-x^2/(4t)}\big)
	$
	as $|x|\to\infty$.
\end{proposition}

\noindent Proposition~\ref{P:flat-cov-upper} yields the Gaussian-tail bound
\eqref{E:flat-cov-upper} used below.

\begin{proof}
	By conditional Jensen and Cauchy--Schwarz,
	\begin{align*}
		\big|\Cov[h(t,x),h(t,0)]\big|
		 & \le \int_0^t\int_{\R}
		\Norm{D_{s,y}h(t,x)}_2\,\Norm{D_{s,y}h(t,0)}_2\,
		\ud y\,\ud s.
	\end{align*}
	Using $D_{s,y}h=D_{s,y}Z/Z$ and \eqref{E:dz-green},
	\[
		D_{s,y}h^{\mathrm{flat}}(t,x)
		=\frac{\mathcal{G}(t,x;s,y)\,Z(s,y)}{Z(t,x)}
		=p_{t-s}(x-y)\,
		\bar{\mathcal{G}}(t,x;s,y)\,
		\frac{Z(s,y)}{Z(t,x)}.
	\]
	Apply H\"older with exponents $6,6,6$ to get
	\[
		\Norm{D_{s,y}h(t,x)}_2
		\le p_{t-s}(x-y)\,
		\Norm{\bar{\mathcal{G}}(t,x;s,y)}_6\,
		\Norm{Z(s,y)}_6\,
		\Norm{Z(t,x)^{-1}}_6.
	\]
	The last three factors are uniformly bounded over $0<s<t$ and $x,y\in\R$ by
	Lemma~\ref{L:barG-moments} and Proposition~\ref{P:Z-moments}. Thus,
	$\Norm{D_{s,y}h(t,x)}_2\le C_t\,p_{t-s}(x-y)$ and similarly
	$\Norm{D_{s,y}h(t,0)}_2\le C_t\,p_{t-s}(y)$. Plugging into the previous
	display yields
	\[
		\big|\Cov[h(t,x),h(t,0)]\big|
		\le C_t\int_0^t\int_{\R} p_{t-s}(x-y)p_{t-s}(y)\,\ud y\,\ud s
		=C_t\int_0^t p_{2(t-s)}(x)\,\ud s
		=C_t\,I_{\mathrm{flat}}(t,x),
	\]
	as desired.
\end{proof}

\paragraph{Identifying the covariance constant.} Motivated by the Gu--Pu style
decomposition \eqref{E:cov-basic}, one is led to expect that for each fixed
$t>0$ there exists $\kappa(t)\in(0,\infty)$ such that
\begin{equation} \label{conj:kpz-cov}
	\Cov[h(t,x),h(t,0)]
	\sim \kappa(t)\,I_{\mathrm{flat}}(t,x),
	\qquad |x|\to\infty,
\end{equation}
where the proxy $I_{\mathrm{flat}}$ is defined in \eqref{E:Iflat-def}. A naive
guess would be $\kappa(t)=1$ (since $\E[\bar{\mathcal{G}}]=1$), but
Theorem~\ref{thm:flat-main} identifies
\[
	\kappa(t)=\Big(\E\!\left[Z(t,0)^{-1}\right]\Big)^2>1,
\]
where the strict inequality follows from Jensen's inequality and non-degeneracy
of $Z(t,0)$.

\begin{lemma}[Removing $Z(s,y)^2$ from the boundary
		layer]\label{L:f-key-assump-Z} Fix $t>0$.
	There exists $C=C(t)<\infty$ such that for all $s\in(0,t\wedge 1]$ and all
	$x,y\in\R$,
	\begin{equation}\label{E:flat-key-assump-Z}
		\left|\E\!\left[\mathcal{X}(t,s,x,y)\right]
		-\E\!\left[\mathcal{B}(t,x;s,y)\,\mathcal{B}(t,0;s,y)\right]\right|
		\le C\, s^{1/4},
	\end{equation}
	where $\mathcal{B}$ is defined in \eqref{E:Bflat-def}.
\end{lemma}
\begin{proof}
	By \eqref{E:Xflat-Bflat},
	\[
		\mathcal{X}(t,s,x,y)
		-\mathcal{B}(t,x;s,y)\mathcal{B}(t,0;s,y)
		=\big(Z(s,y)^2-1\big)\,
		\mathcal{B}(t,x;s,y)\mathcal{B}(t,0;s,y).
	\]
	Use H\"older with exponents $2,4,4$ to bound the $L^1$ norm by
	$\Norm{Z(s,y)^2-1}_2$ times the $L^4$ norms of the two $\mathcal{B}$ factors.
	The first term is $O(s^{1/4})$ by Lemma~\ref{L:Zflat-small} and
	Proposition~\ref{P:Z-moments}. For the $\mathcal{B}$ moments, apply H\"older
	and use Lemma~\ref{L:barG-moments} together with Proposition~\ref{P:Z-moments}
	to bound the negative moments of $Z(t,x)$ uniformly in $x$. This
	yields~\eqref{E:flat-key-assump-Z}.
\end{proof}

\section{Moment and inverse-moment bounds}\label{S:moments}

This section is devoted to the study of the moments and inverse moments of the
solution $Z$ to \eqref{E:she} and of the corresponding Green's function
$\mathcal{G}$. In particular, we require uniform (in space) bounds for the
moments and inverse moments in order to justify that the random normalizations
in the covariance integrand can be replaced by their deterministic counterparts
in the $|x| \to \infty$ asymptotics. Moreover, we derive a closed-form formula
for the second moment of the normalized Green's function $\bar{\mathcal{G}}$
given in~\eqref{E:barG}.

\begin{remark}
	Indeed, $\bar{\calG}_\beta$ satisfies the following SPDE:
	\begin{align}\label{E:barG-SPDE}
		\begin{dcases}
			\partial_t \bar{\calG}_\beta(t,x;s,y) =\tfrac12 \partial_x^2 \bar{\calG}_\beta(t,x;s,y) - \frac{x-y}{t-s} \,  \partial_x \bar{\calG}_\beta(t,x;s,y) + \beta\,\bar{\calG}_\beta(t,x;s,y)\,\xi(t,x), \\
			\bar{\calG}_\beta(s,x;s,y)=1,
		\end{dcases}
	\end{align}
	for $t > s \ge 0$ and $x, y \in \R$; see, e.g.,
	\cite{karatzas.shreve:91:brownian}*{(6.21 on p.~358}
	or~\cite{revuz.yor:99:continuous}*{Exercise (2.12) on p.~384}. This fact will not
	be used though. We will only need the mild formulation in~\eqref{E:barG-mild}
	to derive the moment formulas and bounds.
\end{remark}

\paragraph{Continuum polymer.} The Green's function $\calG_\beta(t,x;s,y)$ is the
point-to-point partition function of the continuum directed random polymer
(CDRP)~\cite{alberts.khanin.ea:14:continuum} from $(s,y)$ to $(t,x)$;
see~\cite{chen.ouyang.ea:26:class} for a recent generalization to~$\R^d$. The
normalized Green's function $\bar{\calG}_\beta$ removes the deterministic
heat-kernel singularity: it admits a continuous extension to $t=s$ with value
$1$, satisfies an SPDE driven by the Brownian-bridge generator, and encodes the
free-energy fluctuations of the polymer endpoint. The systematic study of
$\bar{\calG}_\beta$ as a random field jointly in all parameters was initiated by
Alberts, Janjigian, Rassoul-Agha, and
Sepp\"al\"ainen~\cite{alberts.janjigian.ea:22:greens}, who established joint
continuity, total positivity, and uniform $L^p$ bounds of the form
$\Norm{\bar{\calG}_\beta(t,x;s,y)}_{L^{2p}} \le C_{\beta,p}$;
however, their bounds do not provide explicit dependence on $\beta$, $\tau$, or
the spatial covariance, and no exact moment formula was derived.

\begin{proof}[Proof of Theorem~\ref{thm:barG-second-moment-intro}]
	Note that $\E[\bar{\mathcal{G}}(t,x;s,y)^2]$ depends only on $\tau=t-s$ by \cite{alberts.janjigian.ea:22:greens}*{Proposition~2.3}. Hence, we write $m_2(\tau) \coloneqq
		\E[\bar{\mathcal{G}}(\tau,x;0,0)^2]$. Firstly, using \eqref{E:barG-mild} at
	$(s,y)=(0,0)$ and $(t,x)=(\tau,0)$,
	\begin{equation*}
		\bar{\mathcal{G}}_\beta(\tau,0;0,0)=1
		+\beta\int_0^\tau\int_{\R}
		q_u^{0,\tau}(z)\,
		\bar{\mathcal{G}}_\beta(u,z;0,0)
		\,W(\ud u,\ud z),
	\end{equation*}
	where $q_u^{0,\tau}(z)$ is given in~\eqref{E:bridge-kernel-Rd}. Taking second
	moment, by It\^o isometry and shear invariance, we have
	\begin{equation}
		m_2(\tau)=1+\beta^2\int_0^\tau m_2(u)
		\left(\int_{\R}
		\left(q_u^{0,\tau}(z)\right)^2\,\ud z\right)\ud u.
	\end{equation}
	Now
	\begin{equation}
		\int_{\R}\frac{p_{\tau-u}(z)^2p_u(z)^2}{p_\tau(0)^2}\,\ud z
		=\frac{\tau}{2\pi u(\tau-u)}\int_{\R}
		e^{-z^2\tau/(u(\tau-u))}\,\ud z
		=\frac{\sqrt{\tau}}{2\sqrt{\pi}\sqrt{u(\tau-u)}}.
	\end{equation}
	Therefore, \eqref{E:volterra-1d-intro} holds. Next, define
	\begin{equation}
		a_0(\tau)=1,
		\qquad
		a_{n+1}(\tau)\coloneqq \frac{\beta^2\sqrt{\tau}}{2\sqrt{\pi}}
		\int_0^\tau\frac{a_n(u)}{\sqrt{u(\tau-u)}}\,\ud u, \quad n \ge 0.
	\end{equation}
	Then $m_2(\tau)=\sum_{n\ge 0}a_n(\tau)$ by Picard iteration for
	\eqref{E:volterra-1d-intro}. Here $B(a,b) \coloneqq  \int_0^1 t^{a-1}(1-t)^{b-1}\,\ud t =
		\Gamma(a)\Gamma(b)/\Gamma(a+b)$ denotes the Euler beta function; see
	\cite{olver.lozier.ea:10:nist}*{\S5.12}. The first terms are
	\begin{equation*}
		a_1(\tau)=\frac{\beta^2\sqrt{\tau}}{2\sqrt{\pi}}B\!\left(\frac12,\frac12\right)
		=\frac{\sqrt{\pi\beta^4\tau}}{2},\quad
		a_2(\tau)=\frac{\beta^4\tau}{2}, \quad
		a_3(\tau)=\frac{\beta^6\tau^{3/2}}{4\sqrt{\pi}}B\!\left(\frac32,\frac12\right)
		=\frac{\sqrt{\pi}\,\beta^6\tau^{3/2}}{8}.
	\end{equation*}
	We claim that for all $n\ge 0$,
	\begin{equation}
		a_n(\tau)=c_n(\beta^4\tau)^{n/2}
		\quad \text{and} \quad
		c_n=\frac{\sqrt{\pi}}{2^n\Gamma\!\left(\frac{n+1}{2}\right)}.
	\end{equation}
	The cases $n=0,1,2,3$ are above. If $a_n(\tau)=c_n(\beta^4\tau)^{n/2}$, then
	\begin{equation}
		\begin{aligned}
			a_{n+1}(\tau)
			 & =\frac{\beta^2\sqrt{\tau}}{2\sqrt{\pi}}c_n
			\int_0^\tau (\beta^4 u)^{\frac{n}{2}}u^{-1/2}(\tau-u)^{-1/2}\,\ud u \\
			 & =\frac{c_n}{2\sqrt{\pi}}\,(\beta^4\tau)^{(n+1)/2}
			B\!\left(\frac{n+1}{2},\frac12\right)
			=\frac{c_n}{2}\frac{\Gamma((n+1)/2)}{\Gamma((n+2)/2)}\,(\beta^4\tau)^{(n+1)/2},
		\end{aligned}
	\end{equation}
	which yields
	$c_{n+1}=\frac{\sqrt{\pi}}{2^{n+1}\Gamma\!\left(\frac{n+2}{2}\right)}$. So the
	claim follows by induction. Finally, summing the series,
	\begin{equation*}
		\begin{aligned}
			m_2(\tau)
			=\sum_{n=0}^{\infty} \frac{\sqrt{\pi}}{2^n\Gamma\!\left(\frac{n+1}{2}\right)}\,(\beta^4\tau)^{n/2}
			=\sqrt{\pi}\sum_{n=0}^{\infty} \frac{(\beta^2\sqrt{\tau}/2)^n}{\Gamma\!\left(\frac{n}{2}+\frac12\right)}
			=\sqrt{\pi}\,\ML{1/2}{1/2}\!\left(\frac{\beta^2\sqrt{\tau}}{2}\right),
		\end{aligned}
	\end{equation*}
	where $\ML{\alpha}{\beta}(z)\coloneqq \sum_{n=0}^{\infty} \frac{z^n}{\Gamma(\alpha
			n+\beta)}$ is the two-parameter Mittag-Leffler
	function~\cite{podlubny:99:fractional}. Next, we can conclude the formula
	in~\eqref{E:2nd-Mom-1d-intro} from the following identities:
	\begin{equation*}
		\ML{1/2}{1/2}(z)=\frac{1}{\sqrt{\pi}}+z\ML{1/2}{1}(z), \quad
		\ML{1/2}{1}(z)=e^{z^2}\Erfc(-z), \quad \text{and} \quad
		\Erfc(-z)=2\Phi(z\sqrt{2}).
	\end{equation*}
	This completes the proof of Theorem~\ref{thm:barG-second-moment-intro}.
\end{proof}

\begin{remark}[Comparison with the parabolic Anderson model]
	Consider the parabolic Anderson model with flat initial condition, i.e.
	\[
		\partial_t Z_\beta(t,x)=\tfrac12\partial_x^2 Z_\beta(t,x)
		+ \beta Z_\beta(t,x)\,\xi(t,x),
		\qquad t>0,\ x\in\R,
	\]
	with $Z_\beta(0,x)=1$. Its mild form is given by
	\begin{equation}\label{E:mild-flat}
		Z_\beta(t,x) = 1 + \beta \int_0^t \int_{\R} p_{t-r}(x-z)\, Z_\beta(r,z) \, \xi(\ud r, \ud z).
	\end{equation}
	Comparing \eqref{E:mild-flat} with \eqref{E:barG-mild}, we see that the
	difference is to replace the Brownian bridge kernel $q_{r-s}^{x-y,t-s}(z-y)$
	by the Brownian kernel $p_{t-r}(x-z)$. The second moment of the solution
	to~\eqref{E:mild-flat} has been obtained explicitly
	in~\cite{chen.dalang:15:moments, bertini.cancrini.ea:94:stochastic} as
	follows:
	\begin{align}\label{E:2nd-Mom-1d-Z}
		\E\!\left[Z_\beta(t,x)^2\right]
		=2e^{\beta^4 t/4}\Phi\!\left(\beta^2\sqrt{\frac{t}{2}}\right)
		=e^{\beta^4 t/4}\Erfc\!\left(-\frac{\beta^2\sqrt{t}}{2}\right).
	\end{align}
	See Corollary~2.5 and Remark~2.6 of~\cite{chen.dalang:15:moments}. The
	Volterra integral equation corresponding to~\eqref{E:volterra-1d-intro} is
	\begin{align}\label{E:volterra-flat}
		g(\tau)=1+\frac{\beta^2}{2\sqrt{\pi}}\int_0^\tau\frac{g(s)}{\sqrt{\tau-s}}\,\ud s,
		\qquad g(\tau)\coloneqq \E\!\left[Z_\beta(\tau,x)^2\right].
	\end{align}
	Therefore, the two kernels satisfy
	\[
		K_{\mathrm{bridge}}(t,s) \coloneqq \frac{\beta^2}{2\sqrt{\pi}}
		\frac{\sqrt{t/s}}{\sqrt{t-s}}
		\ge \frac{\beta^2}{2\sqrt{\pi}}\frac{1}{\sqrt{t-s}} \eqqcolon K_{\mathrm{flat}}(t,s), \quad 0<s<t.
	\]
	Using monotone Picard iteration with the same initial iterate $1$, this gives
	$m_2(t)\ge g(t)$ for all $t \ge 0$. This inequality is strict for $t>0$ when
	$\beta\neq 0$. At first order this is already visible:
	\[
		m_2(t)=1+\frac{\sqrt{\pi}}{2}\beta^2\sqrt{t}+O(\beta^4 t),
		\qquad
		g(t)=1+\frac{1}{\sqrt{\pi}}\beta^2\sqrt{t}+O(\beta^4 t),
	\]
	so the bridge-normalized second moment has a larger growth coefficient.
	Intuitively, this happens because dividing by the heat kernel changes the
	propagation from Brownian motion to Brownian bridge. The factor $\sqrt{t/s}$
	upweights early times $s\downarrow 0$, so the normalized Green's function
	accumulates noise more strongly than the flat mild solution.
\end{remark}

Recall that $Z$ is the solution of \eqref{E:mild-flat} when $\beta =1$.

\begin{lemma}[Uniform moments of $\bar{\mathcal{G}}$]\label{L:barG-moments}
	Fix $T>0$ and $k\in\R$. There exists $C=C(T,k)<\infty$ such that
	\begin{equation}\label{E:barG-moments}
		\sup_{\substack{0\le s<t\le T\\ x,y\in\R}}
		\E\!\left[\bar{\mathcal{G}}(t,x;s,y)^k\right]\le C.
	\end{equation}
\end{lemma}
\begin{proof}
	By translation invariance of the white noise, $\bar{\mathcal{G}}(t,x;s,y)$ has
	the same law as $\bar{\mathcal{G}}(t-s,x-y;0,0)$. Thus \eqref{E:barG-moments}
	reduces to uniform bounds on the moments (including negative moments) of
	$\bar{\mathcal{G}}(\tau,z;0,0)$ for $\tau\in(0,T]$ and $z\in\R$. By the shear
	invariance of the normalized Green's function (i.e., stationarity in the
	spatial variable), its law does not depend on $z$, so it suffices to consider
	$z=0$. Such bounds are standard for the one-dimensional SHE Green's function;
	see, for instance, \cite{alberts.janjigian.ea:22:greens}*{Lemma~3.2}.
\end{proof}

\begin{lemma}[Green kernel representation]\label{L:Z-green-repr}
	For flat initial data $u_0$, the mild solution to \eqref{E:she} satisfies, for
	every $t>0$ and $x\in\R$,
	\begin{equation}\label{E:Z-green-repr}
		Z(t,x)=\int_{\R}\mathcal{G}(t,x;0,z) \,\ud z.
	\end{equation}
\end{lemma}
\begin{proof}
	For $t>0$ and $x\in\R$, set $\Phi(t,x)\coloneqq \int_{\R}\mathcal{G}(t,x;0,z)\,\ud z$.
	Using the mild formulation of \eqref{E:barG-mild} (with $s=0$),
	\[
		\mathcal{G}(t,x;0,z)
		=p_t(x-z)+\int_0^t\int_{\R} p_{t-r}(x-w)\,\mathcal{G}(r,w;0,z)\,\xi(\ud r\,\ud w).
	\]
	Integrating in $z$ and using $\int_{\R}p_t(x-z)\,\ud z=1$, together with
	stochastic Fubini (justified, e.g., by truncating the $z$-integral and
	applying It\^o's isometry), yields
	\[
		\Phi(t,x)
		=1+\int_0^t\int_{\R} p_{t-r}(x-w)\,\Phi(r,w)\,\xi(\ud r\,\ud w).
	\]
	This is precisely the mild formulation of \eqref{E:she} with flat initial data
	$u_0\equiv 1$. By uniqueness of mild solutions, $\Phi(t,x)=Z(t,x)$, proving
	\eqref{E:Z-green-repr}.
\end{proof}

\begin{proposition}[Uniform moments of $ Z$]\label{P:Z-moments}
	Fix $T>0$, $p\ge 1$ and $q>0$. Then
	\begin{equation}\label{E:Ztilde-moments}
		\sup_{\substack{0<t\le T\\ x\in\R}}
		\E\!\left[ Z(t,x)^p+ Z(t,x)^{-q}\right]<\infty.
	\end{equation}
\end{proposition}
\begin{proof}
	By Lemma~\ref{L:Z-green-repr}, we have \eqref{E:Z-green-repr}. Combined with the definition of $\bar{\mathcal{G}}$, we obtain
	\[
		Z(t,x)
		=\int_{\R}\bar{\mathcal{G}}(t,x;0,z)\,\nu_{t,x}(\ud z),
	\]
	where $\nu_{t,x}$ is the probability measure
	$
		\nu_{t,x}(\ud z)\coloneqq p_t(x-z)\,\ud z.
	$
	By Jensen, for $p\ge 1$,
	\[
		Z(t,x)^p\le \int_{\R} \bar{\mathcal{G}}(t,x;0,z)^p\,\nu_{t,x}(\ud z),
	\]
	and for $q>0$ (since $r\mapsto r^{-q}$ is convex on $(0,\infty)$),
	\[
		Z(t,x)^{-q}\le \int_{\R} \bar{\mathcal{G}}(t,x;0,z)^{-q}\,\nu_{t,x}(\ud z).
	\]
	Taking expectations and using Lemma~\ref{L:barG-moments} gives
	\eqref{E:Ztilde-moments}.
\end{proof}

\begin{remark}[{The constant $\mathbb{E}[Z(t,0)^{-1}\mathclose]$}]
	\label{R:inverse-moment}
	Proposition~\ref{P:Z-moments} guarantees that $\mathbb{E}[Z(t,0)^{-1}]$ is finite for each $t>0$.
	By the martingale property of $Z(t,\cdot)$ one has $\mathbb{E}[Z(t,0)]=1$, and Jensen's inequality applied to the convex function $x\mapsto 1/x$ on $(0,\infty)$ yields
	\[
		\mathbb{E}[Z(t,0)^{-1}] \ge \frac{1}{\mathbb{E}[Z(t,0)]}=1.
	\]
	Equality would hold only if $Z(t,0)$ were deterministic, which is not the case for the SHE solution.  Hence $\mathbb{E}[Z(t,0)^{-1}]>1$, and consequently the boundary-layer constant $\kappa(t)=\bigl(\mathbb{E}[Z(t,0)^{-1}]\bigr)^2$ is strictly larger than~$1$.

	Quantitative bounds can be obtained from the moment estimates of Proposition~\ref{P:Z-moments}.  For any $q>1$, H\"older's inequality gives
	\[
		\mathbb{E}[Z(t,0)^{-1}] \le \bigl(\mathbb{E}[Z(t,0)^{-q}]\bigr)^{1/q},
	\]
	where the right-hand side is finite and bounded uniformly for $t$ in compact intervals.
	For small times, Lemma~\ref{L:Zflat-small} yields $\Norm{Z(t,0)-1}_k\le C t^{1/4}$; together with the expansion $(1+x)^{-1}=1-x+x^2+O(x^3)$ one finds
	\[
		\mathbb{E}[Z(t,0)^{-1}] = 1 + \mathbb{E}\bigl[(Z(t,0)-1)^2\bigr] + O(t^{3/4}),
	\]
	and $\mathbb{E}[(Z(t,0)-1)^2]\le C t^{1/2}$ by the lemma.  Hence $\mathbb{E}[Z(t,0)^{-1}]=1+O(t^{1/2})$ as $t\to0$.

	The exact value of $\mathbb{E}[Z(t,0)^{-1}]$ for general $t$ appears unknown.  If $\log Z(t,0)$ were Gaussian with variance $\sigma_t^2$, the condition $\mathbb{E}[Z(t,0)]=1$ would force $\mathbb{E}[Z(t,0)^{-1}]=\exp(\sigma_t^2)$.  Although $\log Z(t,0)$ is not exactly Gaussian, this heuristic suggests that $\mathbb{E}[Z(t,0)^{-1}]$ grows with the fluctuations of the solution.  Determining $\kappa(t)$ explicitly, or even its asymptotic behavior as $t\to\infty$, remains an interesting open problem.
\end{remark}

\begin{lemma}[Uniform moment of $\mathcal{X}$]\label{L:X-moment}
	Fix $T>0$. There exists $C=C(T)<\infty$ such that
	\begin{equation}\label{E:X-expect}
		\sup_{\substack{0<s<t\\ x,y\in\R}}
		\E\!\left[\mathcal{X}(t,s,x,y)\right] < \infty.
	\end{equation}
\end{lemma}

\begin{proof}
	This follows from H\"older's inequality:
	\[
		\big|\mathbb{E}\!\left[ \mathcal{X}(t,s,x,y) \right]\big|
		\le \Norm{Z(s,y)^2}_3 \, \Norm{\mathcal{B}(t,x;s,y)}_3 \, \Norm{\mathcal{B}(t,0;s,y)}_3.
	\]
	Using Jensen's inequality together with H\"older's inequality, we obtain
	\[
		\Norm{\mathcal{B}(t,x;s,y)}_3
		\le \Norm{Z(t,x)^{-1}\, \bar{\mathcal{G}}(t,x;s,y)}_3
		\le \Norm{Z(t,x)^{-1}}_6 \, \Norm{\bar{\mathcal{G}}(t,x;s,y)}_6.
	\]
	The claim now follows from Proposition~\ref{P:Z-moments} and Lemma~\ref{L:barG-moments}.
\end{proof}
We will invoke Lemma~\ref{L:X-moment} and \eqref{E:X-expect} when applying dominated convergence in the spatial decay argument.

\begin{lemma}[Small-time estimate for $Z$]\label{L:Zflat-small}
	Fix $k\ge2$ and $T>0$. There exists $C=C(k,T)<\infty$ such that for all
	$s\in(0,T]$ and all $y\in\R$,
	\begin{equation}\label{E:Zflat-small}
		\Norm{Z(s,y)-1}_k\le C\,s^{1/4}.
	\end{equation}
\end{lemma}
\begin{proof}
	By stationarity in the spatial variable, it suffices to consider $y=0$.
	For $s\in(0,1]$, from the mild formulation with flat initial data,
	\[
		Z(s,0)-1
		=\int_0^s\int_{\R} p_{s-r}(-z)\,Z(r,z)\,\xi(\ud r\,\ud z).
	\]
	By BDG, H\"older, and the moment bounds from Proposition~\ref{P:Z-moments}
	(applied with $u_0\equiv 1$, so $Z$), we obtain
	\[
		\Norm{Z(s,0)-1}_k
		\lesssim \left(\int_0^s\int_{\R} p_{s-r}(z)^2\,\Norm{Z(r,z)}_k^2\,\ud z\,\ud r\right)^{1/2}
		\lesssim \left(\int_0^s (s-r)^{-1/2}\,\ud r\right)^{1/2}
		\lesssim s^{1/4},
	\]
	which is \eqref{E:Zflat-small}.
	For $s\in[1,T]$, use the uniform moment bounds from
	Proposition~\ref{P:Z-moments} and the fact that $s^{1/4}\ge 1$ to conclude
	$\Norm{Z(s,0)-1}_k\lesssim s^{1/4}$.
\end{proof}

\begin{lemma}[Continuity of inverse moments]\label{L:Z-inv-cont}
	Fix $t>0$ and $q>0$. Then
	\[
		\lim_{\tau\to t}\ \E\!\left[Z(\tau,0)^{-q}\right]=\E\!\left[Z(t,0)^{-q}\right].
	\]
	In particular, $\E[Z(t-s,0)^{-1}]\to \E[Z(t,0)^{-1}]$ as $s\downarrow 0$.
\end{lemma}
\begin{proof}
	By Proposition~\ref{P:Z-moments}, for any compact interval $[t_0,t_1]\subset(0,\infty)$ one has
	$\sup_{\tau\in[t_0,t_1]}\E[Z(\tau,0)^{-2q}]<\infty$.
	Hence it suffices to prove that $Z(\tau,0)\to Z(t,0)$ in $L^2(\Omega)$ as $\tau\to t$.

	Fix $T>t$ and consider $\tau\in(0,T]$. From the mild formulation with flat initial data,
	\[
		Z(\tau,0)
		=
		1+\int_0^\tau\int_{\R} p_{\tau-r}(-z)\,Z(r,z)\,\xi(\ud r\,\ud z).
	\]
	Let $\tau'\to \tau$. Then
	\[
		Z(\tau',0)-Z(\tau,0)
		=
		\int_0^{\tau\wedge \tau'}\int_{\R}\big(p_{\tau'-r}-p_{\tau-r}\big)(-z)\,Z(r,z)\,\xi(\ud r\,\ud z)
		+\int_{\tau\wedge \tau'}^{\tau\vee \tau'}\int_{\R} p_{\tau^\ast-r}(-z)\,Z(r,z)\,\xi(\ud r\,\ud z),
	\]
	where $\tau^\ast=\tau$ or $\tau'$. By It\^o's isometry and Proposition~\ref{P:Z-moments} (with $p=2$),
	$\sup_{0<r\le T,\ z\in\R}\Norm{Z(r,z)}_2<\infty$, and hence
	\[
		\Norm{Z(\tau',0)-Z(\tau,0)}_2^2
		\lesssim
		\int_0^{\tau\wedge \tau'}\Norm{p_{\tau'-r}-p_{\tau-r}}_{L^2(\R)}^2\,\ud r
		+\int_{\tau\wedge \tau'}^{\tau\vee \tau'} \Norm{p_{\tau^\ast-r}}_{L^2(\R)}^2\,\ud r.
	\]
	Using $\Norm{p_{u}}_{L^2(\R)}^2=(4\pi u)^{-1/2}$ and the continuity of $u\mapsto p_u$ in $L^2(\R)$ for $u>0$,
	the right-hand side tends to $0$ as $\tau'\to\tau$ by dominated convergence.
	Thus $Z(\tau,0)$ is continuous in $L^2(\Omega)$.

	Finally, use
	\[
		\big|x^{-q}-y^{-q}\big|
		\le q\,|x-y|\,(x^{-q-1}+y^{-q-1}),
		\qquad x,y>0,
	\]
	and H\"older (together with Proposition~\ref{P:Z-moments}, applied to $q+1$) to conclude that
	$\E[Z(\tau,0)^{-q}]\to\E[Z(t,0)^{-q}]$ as $\tau\to t$.
\end{proof}

\section{Shear mixing for space--time white noise}\label{SS:shear-mixing}

Throughout this section, we assume that the probability space
$(\Omega,\mathcal{F},\P)$ supports the standard group of measure-preserving
automorphisms of the white-noise space generated by reflections, translations,
shears, dilations, and negation; see \cite{alberts.janjigian.ea:22:greens}
(after equation~(2.1)). For examples of such transformations, see
Example~2.1 therein. In the argument below we primarily use translations and
shears.

We record a basic mixing statement for the shear action on space--time white
noise. This complements the shear invariance identities established in
\cite[Proposition~2.3(iii)]{alberts.janjigian.ea:22:greens}, which yield
stationarity of various fields but do not by themselves imply decay of
correlations as the shear parameter becomes large.

For $(r,\nu)\in\R^2$, let $\theta_{r,\nu}$ denote the shear automorphism of
the white-noise probability space introduced in
\cite[discussion preceding Proposition~2.3]{alberts.janjigian.ea:22:greens}.
These automorphisms act by shearing
space along the time direction. At the level of the Gaussian field, this
transformation is characterized by the identity
\[
	\xi(f)\circ\theta_{r,\nu}
	=
	\xi\!\left(S_{r,-\nu}f\right),
	\qquad f\in L^2(\R^2),
\]
where $\xi(f):=\int_{\R^2} f(t,x)\,\xi(\mathrm dt\,\mathrm dx)$ and $S_{r,\nu}:L^2(\R^2)\to L^2(\R^2)$ is the shear operator defined by
\[
	(S_{r,\nu}f)(t,x)\coloneqq f\bigl(t,x+\nu(t-r)\bigr).
\]
In particular, $S_{r,\nu}$ is a unitary operator on $L^2(\R^2)$ and
$S_{r,\nu}^{-1}=S_{r,-\nu}$.

For $n\ge0$, let
\[
	\mathcal{H}_n = \overline{\text{Span}\{H_n(\xi(f)) : f\in L^2(\R^2),\ \Norm{f}_{L^2(\R^2)}=1\}}^{\,L^2(\Omega)},
\]
where $H_n$ is the
Hermite polynomial of degree $n$ with the normalization of
\cite[Section~1.1]{nualart:06:malliavin}. The space $\mathcal{H}_n$ is called the
$n$-th Wiener chaos.


\begin{proposition}\label{P:theta-invariant-Hn}
	For every $n\ge0$ and every $(r,\nu)\in\R^2$, the shear automorphism
	$\theta_{r,\nu}$ preserves the $n$-th Wiener chaos $\mathcal{H}_n$, i.e.
	if $F\in\mathcal{H}_n$, then $F\circ\theta_{r,\nu}\in\mathcal{H}_n$.
\end{proposition}

\begin{proof}
	For each generator $H_n(\xi(f))$ with $\Norm{f}_{L^2(\R^2)}=1$,
	the defining property $\xi(f)\circ\theta_{r,\nu}=\xi(S_{r,-\nu}f)$ gives
	\[
		H_n(\xi(f))\circ\theta_{r,\nu}
		= H_n\bigl(\xi(S_{r,-\nu}f)\bigr)
		\in \mathcal{H}_n,
	\]
	since $S_{r,-\nu}$ is unitary and hence
	$\Norm{S_{r,-\nu}f}_{L^2(\R^2)}=1$. The claim extends to all of
	$\mathcal{H}_n$ by linearity and $L^2(\Omega)$-closure
	(using that composition with the measure-preserving map
	$\theta_{r,\nu}$ is an isometry on $L^2(\Omega)$).
\end{proof}

\begin{remark}
	The composition $F\circ\theta_{r,\nu}$ is well defined for every
	$F\in L^2(\Omega)$. Indeed, it is well known that
	$L^2(\Omega)=\bigoplus_{n=0}^{\infty}\mathcal{H}_n$.
	Hence any $F\in L^2(\Omega)$ admits an expansion
	$F=\sum_{n=0}^{\infty}F_n$ with $F_n\in\mathcal{H}_n$ and convergence in
	$L^2(\Omega)$. Since Proposition~\ref{P:theta-invariant-Hn} shows that
	$\theta_{r,\nu}$ preserves each chaos $\mathcal{H}_n$, we may define
	\[
		F\circ\theta_{r,\nu}\coloneqq \sum_{n=0}^{\infty}F_n\circ\theta_{r,\nu},
		\qquad\text{in }L^2(\Omega).
	\]
	Because $\theta_{r,\nu}$ is measure-preserving, composition with
	$\theta_{r,\nu}$ is an isometry on $L^2(\Omega)$. Hence, if
	$S_N\coloneqq \sum_{n=0}^{N}F_n$, then
	$\Norm{S_N\circ\theta_{r,\nu}-S_M\circ\theta_{r,\nu}}_2
		=\Norm{S_N-S_M}_2\to0$ as $M,N\to\infty$, so the series above
	converges in $L^2(\Omega)$. Thus $F\circ\theta_{r,\nu}$ is well defined for
	all $F\in L^2(\Omega)$.
\end{remark}

\begin{lemma}[Shear is strongly mixing]\label{L:shear-mixing}
	For every $r\in\R$ and all $F,G\in L^2(\Omega)$,
	\[
		\lim_{|\nu|\to\infty} \E\!\left[F\cdot(G\circ\theta_{r,\nu})\right]
		= \E[F]\,\E[G];
	\]
	equivalently, $\Cov(F,\,G\circ\theta_{r,\nu})\to0$ as $|\nu|\to\infty$.
\end{lemma}
\begin{proof}
	This lemma will be proved in four steps:

	\paragraph{Step I: weak convergence of the shear operators.} We first consider
	the case $F=\xi(f)$ and $G=\xi(g)$, where $f,g\in L^2(\R^2)$. Since
	$\xi(g)\circ\theta_{r,\nu}=\xi(S_{r,-\nu}g),$ it is enough to prove that
	\begin{equation}\label{E:Uweak0}
		\langle f,S_{r,-\nu}g\rangle_{L^2(\R^2)}\longrightarrow 0,
		\qquad |\nu|\to\infty .
	\end{equation}

	Let $f_k,g_k\in L^2(\R^2)$ be compactly supported and satisfy
	$f_k\to f$ and $g_k\to g$ in $L^2(\R^2)$. Since $S_{r,-\nu}$ is unitary,
	\[
		\begin{aligned}
			\bigl|\langle f,S_{r,-\nu}g\rangle-\langle f_k,S_{r,-\nu}g_k\rangle\bigr|
			 & \le \bigl|\langle f-f_k,S_{r,-\nu}g\rangle\bigr|
			+ \bigl|\langle f_k,S_{r,-\nu}(g-g_k)\rangle\bigr|    \\
			 & \le \Norm{f-f_k}_{L^2(\R^2)}\,\Norm{g}_{L^2(\R^2)}
			+ \Norm{f_k}_{L^2(\R^2)}\,\Norm{g-g_k}_{L^2(\R^2)},
		\end{aligned}
	\]
	uniformly in $\nu$. Therefore it suffices to prove \eqref{E:Uweak0} for
	compactly supported $f,g$.

	So let $f,g\in L^2(\R^2)$ be compactly supported. Then
	\[
		\langle f,S_{r,-\nu}g\rangle
		=
		\int_{\R}\int_{\R} f(t,x)\,g(t,x-\nu(t-r))\,\ud x\,\ud t.
	\]
	Fix $\varepsilon>0$, and split the $t$-integral over the sets
	\[
		A_{\varepsilon}\coloneqq \{|t-r|\le \varepsilon\},
		\qquad
		A_{\varepsilon}^c\coloneqq \R\setminus A_{\varepsilon}.
	\]
	On $A_{\varepsilon}$, by Cauchy--Schwarz,
	\[
		\left|
		\int_{A_{\varepsilon}}\int_{\R}
		f(t,x)\,g(t,x-\nu(t-r))\,\ud x\,\ud t
		\right|
		\le
		\Norm{f}_{L^2(\R^2)}\,
		\Norm{g\cdot \mathbf 1_{A_{\varepsilon}}}_{L^2(\R^2)}.
	\]
	Since $\mathbf 1_{A_{\varepsilon}}\to0$ pointwise a.e. as
	$\varepsilon\downarrow0$ and $|g\cdot \mathbf 1_{A_{\varepsilon}}|^2\le
		|g|^2\in L^1(\R^2)$, dominated convergence gives $\Norm{g\cdot \mathbf
			1_{A_{\varepsilon}}}_{L^2(\R^2)}\to0$. Fix such an $\varepsilon$.

	On $A_{\varepsilon}^c$, one has $|\nu(t-r)|\ge |\nu|\varepsilon$. Since $f,g$
	are compactly supported in $\R^2$, there exist $R_f,R_g<\infty$ such that
	$f(t,\cdot)$ and $g(t,\cdot)$ vanish outside $[-R_f,R_f]$ and $[-R_g,R_g]$,
	respectively, for every $t$. For each fixed $t\in A_{\varepsilon}^c$, the
	inner integral
	\[
		\int_{\R} f(t,x)\,g(t,x-\nu(t-r))\,\ud x
	\]
	is the $L^2(\R)$ inner product of $f(t,\cdot)$ with a translate of
	$g(t,\cdot)$ by an amount of magnitude at least $|\nu|\varepsilon$. If
	$|\nu|\varepsilon>R_f+R_g$, then these two supports are disjoint, so the inner
	product is identically $0$, uniformly in $t\in A_{\varepsilon}^c$. Hence the
	contribution from $A_{\varepsilon}^c$ vanishes for $|\nu|\gg1$, which proves
	\eqref{E:Uweak0}.

	\paragraph{Step II: the case of the $n$-th Wiener chaos.} Fix $n\ge1$. We
	prove the desired limit for random variables in $\mathcal{H}_n$. Since
	$\mathcal{H}_n$ is the closure in $L^2(\Omega)$ of the linear span of random
	variables of the form $H_n(\xi(f))$ with $f\in L^2(\R^2)$ and
	$\Norm{f}_{L^2(\R^2)}=1$, it is enough, by density, to establish the claim for
	finite linear combinations of such generators.

	Accordingly, let
	\[
		F=\sum_{i=1}^{N} a_i\,H_n(\xi(f_i))
		\quad \text{and} \quad
		G=\sum_{j=1}^{M} b_j\,H_n(\xi(g_j)),
	\]
	where $f_i,g_j\in L^2(\R^2)$ satisfy $\Norm{f_i}_{L^2(\R^2)}=\Norm{g_j}_{L^2(\R^2)}=1$, and
	$a_i,b_j\in\R$.

	Since $\xi(g)\circ\theta_{r,\nu}=\xi(S_{r,-\nu}g)$, we have
	\[
		G\circ\theta_{r,\nu}
		=
		\sum_{j=1}^{M} b_j\,H_n\!\bigl(\xi(S_{r,-\nu}g_j)\bigr).
	\]
	Hence,
	\[
		\E\bigl[F(G\circ\theta_{r,\nu})\bigr]
		=
		\sum_{i=1}^{N}\sum_{j=1}^{M}
		a_i b_j\,
		\E\!\Big[
			H_n(\xi(f_i))\,H_n\bigl(\xi(S_{r,-\nu}g_j)\bigr)
			\Big].
	\]
	Since $\Norm{f_i}_{L^2(\R^2)}=\Norm{S_{r,-\nu}g_j}_{L^2(\R^2)}=1$, the pair
	$\bigl(\xi(f_i),\xi(S_{r,-\nu}g_j)\bigr)$ is jointly Gaussian with unit
	variances. By Lemma~1.1.1 in \cite{nualart:06:malliavin},
	\[
		\E\!\Big[
			H_n(\xi(f_i))\,H_n\bigl(\xi(S_{r,-\nu}g_j)\bigr)
			\Big]
		=
		n!\,\ip{f_i}{S_{r,-\nu}g_j}_{L^2(\R^2)}^{\,n}.
	\]
	Therefore,
	\[
		\E\bigl[F(G\circ\theta_{r,\nu})\bigr]
		=
		n! \sum_{i=1}^{N}\sum_{j=1}^{M}
		a_i b_j\,\ip{f_i}{S_{r,-\nu}g_j}_{L^2(\R^2)}^{\,n}.
	\]

	By Step~I, $\ip{f_i}{S_{r,-\nu}g_j}\to 0$ as $|\nu|\to\infty$ for every $i,j$.
	Since this sum is finite, it follows that
	\[
		\E\bigl[F(G\circ\theta_{r,\nu})\bigr]\longrightarrow0
		\qquad\text{as }|\nu|\to\infty.
	\]
	Now let $F,G\in\mathcal{H}_n$ be arbitrary. Choose finite linear combinations
	$F^{(k)},G^{(k)}$ of generators $H_n(\xi(h))$ with $\Norm{h}_{L^2(\R^2)}=1$
	such that $F^{(k)}\to F$ and $G^{(k)}\to G$ in $L^2(\Omega)$. Then
	\[
		\begin{aligned}
			\left|\E\!\left[F(G\circ\theta_{r,\nu})\right]-\E\!\left[F^{(k)}(G^{(k)}\circ\theta_{r,\nu})\right]\right|
			\le \Norm{F-F^{(k)}}_2\,\Norm{G}_2
			+ \Norm{F^{(k)}}_2\,\Norm{G-G^{(k)}}_2,
		\end{aligned}
	\]
	where we used that composition with $\theta_{r,\nu}$ is an isometry on
	$L^2(\Omega)$. The right-hand side is independent of $\nu$ and tends to $0$ as
	$k\to\infty$ because $\Norm{F^{(k)}}_2\le
		\Norm{F}_2+\Norm{F-F^{(k)}}_2$. Since the previous
	paragraph gives $\E\!\left[F^{(k)}(G^{(k)}\circ\theta_{r,\nu})\right]\to0$ for
	each fixed $k$, it follows that
	\[
		\E\bigl[F(G\circ\theta_{r,\nu})\bigr]\longrightarrow0
		\qquad\text{as }|\nu|\to\infty
	\]
	for all $F,G\in\mathcal{H}_n$.

	\paragraph{Step III: finite chaos expansions.} Let
	\[
		F=\sum_{n=0}^{N} F_n
		\quad \text{and} \quad
		G=\sum_{m=0}^{M} G_m,
	\]
	where $F_n\in \mathcal{H}_n,$ and $G_m\in \mathcal{H}_m .$ By
	Proposition~\ref{P:theta-invariant-Hn}, we have $ G_m\circ\theta_{r,\nu}\in
		\mathcal{H}_m $ for every $m \ge 0$. Therefore, by orthogonality of Wiener
	chaoses,
	\[
		\E\!\left[F_n(G_m\circ\theta_{r,\nu})\right]=0
		\qquad\text{whenever }n\neq m .
	\]
	Hence
	\[
		\E\!\left[F(G\circ\theta_{r,\nu})\right]
		=
		\sum_{n=0}^{N}\sum_{m=0}^{M}
		\E\!\left[F_n(G_m\circ\theta_{r,\nu})\right]
		=
		\sum_{n=0}^{\min\{N,M\}}
		\E\!\left[F_n(G_n\circ\theta_{r,\nu})\right].
	\]

	For every $n\ge1$, Step~II shows that
	\[
		\E\!\left[F_n(G_n\circ\theta_{r,\nu})\right]\longrightarrow 0
		\qquad\text{as }|\nu|\to\infty .
	\]
	For $n=0$, $F_0=\E[F]$ and $G_0=\E[G]$ are constants, so
	\[
		\E\!\left[F_0(G_0\circ\theta_{r,\nu})\right]=\E[F]\E[G].
	\]
	Therefore
	\[
		\E\!\left[F(G\circ\theta_{r,\nu})\right]
		\longrightarrow
		\E[F]\E[G]
		\qquad\text{as }|\nu|\to\infty
	\]
	for all finite chaos expansions $F$ and $G$.

	\paragraph{Step IV: $L^2$ approximation.} Since
	$L^2(\Omega)=\bigoplus_{n=0}^{\infty}\mathcal H_n$, every $F,G\in L^2(\Omega)$
	can be approximated in $L^2(\Omega)$ by finite chaos expansions of the form
	considered in Step~III. Thus there exist finite chaos expansions
	$F^{(k)},G^{(k)}$ such that
	\[
		\Norm{F^{(k)}-F}_2\to 0
		\quad \text{and} \quad
		\Norm{G^{(k)}-G}_2\to 0.
	\]
	Writing
	\[
		\begin{aligned}
			\E\!\left[F(G\circ\theta_{r,\nu})\right]
			-\E\!\left[F^{(k)}(G^{(k)}\circ\theta_{r,\nu})\right]
			 & =
			\E\!\left[(F-F^{(k)})(G\circ\theta_{r,\nu})\right]  +
			\E\!\left[F^{(k)}\bigl((G-G^{(k)})\circ\theta_{r,\nu}\bigr)\right].
		\end{aligned}
	\]
	Using Cauchy--Schwarz, we obtain
	\[
		\left|
		\E\!\left[F(G\circ\theta_{r,\nu})\right]
		-
		\E\!\left[F^{(k)}(G^{(k)}\circ\theta_{r,\nu})\right]
		\right|
		\le
		\Norm{F-F^{(k)}}_2\,\Norm{G}_2
		+
		\Norm{F^{(k)}}_2\,\Norm{G-G^{(k)}}_2.
	\]
	The right-hand side is independent of $\nu$ and tends to $0$ as $k\to\infty$
	because $\Norm{F^{(k)}}_2\le
		\Norm{F}_2+\Norm{F-F^{(k)}}_2$. For fixed $k$, Step~III
	gives
	\[
		\E\!\left[F^{(k)}(G^{(k)}\circ\theta_{r,\nu})\right]
		\longrightarrow
		\E[F^{(k)}]\E[G^{(k)}]
		\qquad\text{as }|\nu|\to\infty.
	\]
	Letting first $|\nu|\to\infty$ and then $k\to\infty$, we obtain
	$\E\!\left[F(G\circ\theta_{r,\nu})\right] \to \E[F]\E[G]$. This
	proves Lemma~\ref{L:shear-mixing}.
\end{proof}

\section{Spatial decay: boundary layer approach}\label{S:spatial-decay}
Throughout this section, constants denoted by $C$ and $C_t$ are finite positive
constants that may change from line to line; when we write $C_t$, it depends
only on $t$ (and on fixed moment exponents), but never on the space-time
variables $s,x,y$.
\begin{lemma}[Exponential negligibility away from $s=0$]\label{L:t-t}
	Fix $t>0$ and $L\in(0,t)$. Then, as $|x|\to\infty$,
	\begin{equation}\label{E:time-tail}
		\int_{L}^{t}\int_{\R}
		\E\!\left[\mathcal{X}(t,s,x,y)\right]\,
		K(t,s,x,y)\,\ud y\,\ud s
		=o\!\left(|x|^{-2}\exp\!\left(-\frac{x^2}{4t}\right)\right).
	\end{equation}
\end{lemma}
\begin{proof}
	By Lemma~\ref{L:X-moment}, taking $T=t$,
	\[
		\int_{L}^{t}\int_{\R}
		\E\!\left[\mathcal{X}(t,s,x,y)\right]K(t,s,x,y)\,\ud y\,\ud s
		\le C_t \int_{L}^{t}  p_{2(t-s)}(x)\,\ud s.
	\]
	For $s\in[L,t)$, $p_{2(t-s)}(x)\le (4\pi (t-s))^{-1/2}\exp(-x^2/(4(t-L)))$,
	and hence,
	\[
		\int_{L}^{t}  p_{2(t-s)}(x)\,\ud s
		\le C_{t,L}\exp\!\left(-\frac{x^2}{4(t-L)}\right)
		= C_{t,L} e^{-x^2/(4t)}e^{-c x^2}, \quad \text{with $c=c(t,L)= \frac{L}{ 4 t( t- L )}.$}
	\]
	This completes the proof of \eqref{E:time-tail}.
\end{proof}
\begin{lemma}[Boundary-layer scale $s\lesssim |x|^{-2}$]\label{L:t-t-x2}
	Fix $t>0$ and $\varepsilon\in(0,2)$. Then, as $|x|\to\infty$,
	\begin{equation}\label{E:time-tail-x2}
		\int_{|x|^{-2+\varepsilon}}^{t}\int_{\R}
		\E\!\left[\mathcal{X}(t,s,x,y)\right]\,
		K(t,s,x,y)\,\ud y\,\ud s
		=o\!\left(|x|^{-2}\exp\!\left(-\frac{x^2}{4t}\right)\right).
	\end{equation}
\end{lemma}
\begin{proof}
	The argument is the same as in Lemma~\ref{L:t-t}, taking $L_x
		=|x|^{-2+\varepsilon}$. As $|x| \to \infty$, note that $c(t,L_x) \ge
		\frac{L_x}{8 t^2}$. Hence, as $|x| \to \infty$,  $\exp(-c x^2) \le
		\exp\left(- |x|^{\varepsilon} / (8 t^2 ) \right)$. This completes the proof
	of \eqref{E:time-tail-x2}.
\end{proof}

\begin{lemma}\label{L:Bcov0}
	Fix $t>0$ and $\varepsilon\in(0,2)$. As $|x|\to\infty$,
	\begin{equation}\label{E:Bcov0}
		\int_{0}^{|x|^{-2+\varepsilon}}\!\!\int_{\R}
		\Big(
		\E\!\left[\mathcal{X}(t,s,x,y)\right]
		- \E\!\left[\mathcal{B}(t,x;s,y)\,\mathcal{B}(t,0;s,y)\right]
		\Big)\,
		K(t,s,x,y)\,\ud y\,\ud s
		=o\!\left(|x|^{-2}e^{-\frac{x^2}{4t}}\right).
	\end{equation}
\end{lemma}
\begin{proof}
	This follows immediately from Lemma~\ref{L:f-key-assump-Z}
	and~\eqref{E:Iflat-asymp}. Denote by $I$ the integral in \eqref{E:Bcov0} and
	set $L_x\coloneqq |x|^{-2+\varepsilon}$. Then
	\[
		|I|
		\lesssim \int_{0}^{L_x}\int_{\R} s^{1/4}\,
		K(t,s,x,y)\,\ud y\,\ud s
		= \int_{0}^{L_x} s^{1/4}\,p_{2(t-s)}(x)\,\ud s
		\le L_x^{1/4}\int_{0}^{L_x} p_{2(t-s)}(x)\,\ud s
		\le L_x^{1/4}\,I_{\mathrm{flat}}(t,x).
	\]
	Using \eqref{E:Iflat-asymp}, we have
	$I_{\mathrm{flat}}(t,x)=O\!\left(|x|^{-2}\exp\!\left(-{x^2}/{(4t)}\right)\right)$
	as $|x|\to\infty$, while $L_x^{1/4}=|x|^{(-2+\varepsilon)/4}=o(1)$ for
	$\varepsilon\in(0,2)$. This yields \eqref{E:Bcov0}.
\end{proof}

Fix $\varepsilon\in(0,2)$. As a consequence of Lemmas~\ref{L:t-t-x2}
and~\ref{L:Bcov0}, we get: as $|x| \to \infty$,
\begin{equation}\label{E:r-BB}
	\Cov[h(t,x),h(t,0)]
	=  \int_{0}^{|x|^{-2+\varepsilon}}\!\!\!\int_{\R}  \E\!\left[\mathcal{B}(t,x;s,y)\,\mathcal{B}(t,0;s,y)\right] \,
	K(t,s,x,y)\,\ud y\,\ud s + o\!\left(|x|^{-2}e^{ -\frac{x^2}{4t} }\right).
\end{equation}

We now replace the boundary-layer factor $\mathcal{B}$ by a future-only analogue
which is deterministic and therefore amenable to shear-mixing arguments. Fix
$t>0$. For $0<s<t$ define the \emph{future-only} flat partition function
\begin{equation}\label{E:Ztilde}
	\widetilde Z(t,x;s)\coloneqq \int_{\R}\mathcal{G}(t,x;s,u)\,\ud u,
\end{equation}
which depends only on the noise in the time interval $(s,t]$. Define also
\begin{equation}\label{E:Btilde}
	\widetilde{\mathcal{B}}(t,x;s,y)
	\coloneqq \E_{s,t}\!\left[\frac{\bar{\mathcal{G}}(t,x;s,y)}{\widetilde Z(t,x;s)}\right]
	= \E\left[\frac{\bar{\mathcal{G}}(t,x;s,y)}{\widetilde Z(t,x;s)}\right],
\end{equation}
where the equality is due to the fact that both $\widetilde Z(t,x;s)$ and
$\bar{\mathcal{G}}(t,x;s,y)$ depend only on the noise in $(s,t]$ and hence
$\widetilde{\mathcal{B}}(t,x;s,y)$ is deterministic.

\begin{lemma}[Small-time reduction of $\mathcal{B}$]\label{L:B-s-r}
	There exists $C_t<\infty$ such that for all $s\in(0,t\wedge 1]$ and all
	$x,y\in\R$,
	\[
		\Norm{\mathcal{B}(t,x;s,y)-\widetilde{\mathcal{B}}(t,x;s,y)}_2
		\le C_t\,s^{1/4}.
	\]
\end{lemma}
\begin{proof}
	Fix $s\in(0,1]$. By the Markov property of the SHE~\eqref{E:she},
	$Z(t,x)=\int_{\R}\mathcal{G}(t,x;s,u)\,Z(s,u)\,\ud u$. Therefore,
	\[
		Z(t,x)-\widetilde Z(t,x;s)
		=\int_{\R}\mathcal{G}(t,x;s,u)\big(Z(s,u)-1\big)\,\ud u.
	\]
	Let $k\ge 2$. By Minkowski and the independence between $\mathcal{G}(t,x;s,u)$
	and $Z(s,u)$,
	\[
		\begin{aligned}
			\Norm{Z(t,x)-\widetilde Z(t,x;s)}_k
			 & \le \int_{\R}\Norm{\mathcal{G}(t,x;s,u)\big(Z(s,u)-1\big)}_k\,\ud u \\
			 & =\int_{\R}\Norm{\mathcal{G}(t,x;s,u)}_k\,\Norm{Z(s,u)-1}_k\,\ud u.
		\end{aligned}
	\]
	By Lemma~\ref{L:barG-moments} and $\mathcal{G}=p\,\bar{\mathcal{G}}$, we have
	$\Norm{\mathcal{G}(t,x;s,u)}_k\le C_t\,p_{t-s}(x-u)$, uniformly in $x,u$.
	Moreover, Lemma~\ref{L:Zflat-small} yields $\Norm{Z(s,u)-1}_k\le C_t\,s^{1/4}$
	uniformly in $u$. Using $\int_{\R}p_{t-s}(x-u)\,\ud u=1$, we obtain
	\begin{equation}\label{E:Z-vs-Ztilde-smalltime}
		\sup_{x\in\R}\Norm{Z(t,x)-\widetilde Z(t,x;s)}_k
		\le C_t\,s^{1/4}.
	\end{equation}

	Next, using the identity
	\[
		\left|Z(t,x)^{-1}-\widetilde Z(t,x;s)^{-1}\right|
		=\frac{\left|Z(t,x)-\widetilde Z(t,x;s)\right|}{Z(t,x)\,\widetilde Z(t,x;s)}.
	\]
	Apply H\"older and Proposition~\ref{P:Z-moments} (with $T=t$ and $q=9$): since
	$\widetilde Z(t,x;s)$ has the same law as $Z(t-s,0)$ by stationarity of the noise, for $k=3$,
	\[
		\sup_{x\in\R}\Norm{Z(t,x)^{-1}-\widetilde Z(t,x;s)^{-1}}_3
		\le
		\sup_{x\in\R}\Norm{Z(t,x)-\widetilde Z(t,x;s)}_9\,
		\sup_{x\in\R}\Norm{Z(t,x)^{-1}}_9\,
		\sup_{x\in\R}\Norm{\widetilde Z(t,x;s)^{-1}}_9
		\le C_t\,s^{1/4},
	\]
	by \eqref{E:Z-vs-Ztilde-smalltime}.
	Therefore, using conditional Jensen and H\"older,
	\[
		\begin{aligned}
			\Norm{\mathcal{B}(t,x;s,y)-\widetilde{\mathcal{B}}(t,x;s,y)}_2
			 & =\Norm{\E_{s,t}\!\left[\bar{\mathcal{G}}(t,x;s,y)\left(Z(t,x)^{-1}-\widetilde Z(t,x;s)^{-1}\right)\right]}_2 \\
			 & \le \Norm{\bar{\mathcal{G}}(t,x;s,y)\left(Z(t,x)^{-1}-\widetilde Z(t,x;s)^{-1}\right)}_2                     \\
			 & \le \Norm{\bar{\mathcal{G}}(t,x;s,y)}_6\,\Norm{Z(t,x)^{-1}-\widetilde Z(t,x;s)^{-1}}_3
			\le C_t\,s^{1/4},
		\end{aligned}
	\]
	where we used Lemma~\ref{L:barG-moments} for $\Norm{\bar{\mathcal{G}}}_6$.
\end{proof}

\begin{lemma}[Product reduction to the future-only profile]\label{L:key-a-to-Btilde}
	Fix $t>0$. Then there exists $C_t<\infty$ such that for all $s\in(0,t\wedge
		1]$ and all $x,y\in\R$,
	\begin{equation}\label{E:key-a-to-Btilde}
		\left|
		\E\!\left[\mathcal{B}(t,x;s,y)\mathcal{B}(t,0;s,y)\right]
		-\widetilde{\mathcal{B}}(t,x;s,y)\,\widetilde{\mathcal{B}}(t,0;s,y)
		\right|
		\le C_t\,s^{1/4}.
	\end{equation}
\end{lemma}
\begin{proof}
	By H\"older and Lemma~\ref{L:barG-moments} together with
	Proposition~\ref{P:Z-moments}, there exists $C_t<\infty$ such that
	\begin{equation}\label{E:B-unif-L2}
		\sup_{\substack{0<s<t\\ x,y\in\R}}\Norm{\mathcal{B}(t,x;s,y)}_2
		\le \sup_{\substack{0<s<t\\ x,y\in\R}}\Norm{\bar{\mathcal{G}}(t,x;s,y)}_4\,\sup_{\substack{0<s<t\\ x\in\R}}\Norm{Z(t,x)^{-1}}_4
		\le C_t.
	\end{equation}
	Similarly, by Lemma~\ref{L:barG-moments} and Proposition~\ref{P:Z-moments}
	(applied to $\widetilde Z(t,x;s)$, which has the same law as $Z(t-s,0)$),
	\begin{equation}\label{E:Btilde-unif}
		\sup_{\substack{0<s<t\\ x,y\in\R}}|\widetilde{\mathcal{B}}(t,x;s,y)|<\infty.
	\end{equation}

	Now write $\mathcal{B}_x\coloneqq \mathcal{B}(t,x;s,y)$, $\mathcal{B}_0\coloneqq \mathcal{B}(t,0;s,y)$,
	$\widetilde{\mathcal{B}}_x\coloneqq \widetilde{\mathcal{B}}(t,x;s,y)$ and $\widetilde{\mathcal{B}}_0\coloneqq \widetilde{\mathcal{B}}(t,0;s,y)$.
	Then
	\[
		\E[\mathcal{B}_x\mathcal{B}_0]-\widetilde{\mathcal{B}}_x\widetilde{\mathcal{B}}_0
		=\E\!\left[(\mathcal{B}_x-\widetilde{\mathcal{B}}_x)\mathcal{B}_0\right]
		+\widetilde{\mathcal{B}}_x\,\E[\mathcal{B}_0-\widetilde{\mathcal{B}}_0].
	\]
	The first term is bounded by $\Norm{\mathcal{B}_x-\widetilde{\mathcal{B}}_x}_2\,\Norm{\mathcal{B}_0}_2$,
	and the second by $|\widetilde{\mathcal{B}}_x|\,\Norm{\mathcal{B}_0-\widetilde{\mathcal{B}}_0}_2$.
	Using Lemma~\ref{L:B-s-r} and the uniform bounds \eqref{E:B-unif-L2}--\eqref{E:Btilde-unif} yields \eqref{E:key-a-to-Btilde}.
\end{proof}

\begin{lemma}[Annealed shift identity for flat initial data]\label{L:B-shift}
	Fix $0<s<t$ and
	$x,y\in\R$. By translation stationarity of the noise and then
	shear invariance of the normalized Green's
	function~\cite{alberts.janjigian.ea:22:greens}*{Proposition~2.3(iii)},
	\begin{align}
		\E\!\left[\mathcal{B}(t,x;s,y)\right]
		 & = \E\!\left[
			\frac{\bar{\mathcal{G}}(t,0;s,y-x)}{\int_{\R} p_{t-s}(z)\,\bar{\mathcal{G}}(t,0;s,z)\,Z(s,z)\,\ud z}
		\right] \label{E:B-shift} \\
		 & = \E\!\left[
			\frac{\bar{\mathcal{G}}(t,0;s,0)}{\int_{\R} p_{t-s}(z+y-x)\,\bar{\mathcal{G}}(t,0;s,z)\,Z(s,z)\,\ud z}
			\right]. \label{E:B-shift-shear}
	\end{align}
	Similarly, for the future-only profile,
	\begin{align}
		\widetilde{\mathcal{B}}(t,x;s,y)
		 & = \E\!\left[
			\frac{\bar{\mathcal{G}}(t,0;s,y-x)}{\int_{\R} p_{t-s}(z)\,\bar{\mathcal{G}}(t,0;s,z)\,\ud z}
		\right] \label{E:Btilde-shift} \\
		 & = \E\!\left[
			\frac{\bar{\mathcal{G}}(t,0;s,0)}{\int_{\R} p_{t-s}(z+y-x)\,\bar{\mathcal{G}}(t,0;s,z)\,\ud z}
			\right]. \label{E:Btilde-shift-shear}
	\end{align}
\end{lemma}

\begin{proof}
	By the Markov property of the SHE,
	\[
		Z(t,x)=\int_{\R}\mathcal{G}(t,x;s,u)\,Z(s,u)\,\ud u
		=\int_{\R} p_{t-s}(x-u)\,\bar{\mathcal{G}}(t,x;s,u)\,Z(s,u)\,\ud u.
	\]
	Thus, by definition \eqref{E:Bflat-def},
	\[
		\E\!\left[\mathcal{B}(t,x;s,y)\right]
		=
		\E\!\left[
			\E_{s,t}\!\left[
				\frac{\bar{\mathcal{G}}(t,x;s,y)}{\int_{\R} p_{t-s}(x-u)\,\bar{\mathcal{G}}(t,x;s,u)\,Z(s,u)\,\ud u}
				\right]
			\right].
	\]
	By translation invariance of the space--time white noise on $(s,t]$, the
	$\E_{s,t}$-expectation is unchanged if we shift the spatial coordinates by
	$-x$. Using also the symmetry of $p_{t-s}$ and changing variables $z=u-x$, we
	obtain
	\[
		\E_{s,t}\!\left[
			\frac{\bar{\mathcal{G}}(t,x;s,y)}{\int_{\R} p_{t-s}(x-u)\,\bar{\mathcal{G}}(t,x;s,u)\,Z(s,u)\,\ud u}
			\right]
		=
		\E_{s,t}\!\left[
			\frac{\bar{\mathcal{G}}(t,0;s,y-x)}{\int_{\R} p_{t-s}(z)\,\bar{\mathcal{G}}(t,0;s,z)\,Z(s,z+x)\,\ud z}
			\right].
	\]
	Since the field $\{Z(s,z+x)\}_{z\in\R}$ is $\mathcal{F}_s$-measurable, it is
	independent of the noise on $(s,t]$. Hence the inner $\E_{s,t}$-expectation
	is a measurable functional of $\{Z(s,z+x)\}_{z\in\R}$. By spatial
	stationarity of $\{Z(s,z)\}_{z\in\R}$ under flat initial data, the outer
	expectation is unchanged if we replace $\{Z(s,z+x)\}_{z\in\R}$ by
	$\{Z(s,z)\}_{z\in\R}$ in law. This yields \eqref{E:B-shift}.

	Next, set $a\coloneqq y-x$. By the shear invariance
	in~\cite{alberts.janjigian.ea:22:greens}*{Proposition~2.3(iii)} (applied to
	the normalized Green's function, with $r=t$), under $\E_{s,t}$ the process
	$z\mapsto \bar{\mathcal{G}}(t,0;s,z)$ is stationary in $z$. Therefore the
	ratio inside the previous display has the same law as
	\[
		\frac{\bar{\mathcal{G}}(t,0;s,0)}{\int_{\R} p_{t-s}(z)\,\bar{\mathcal{G}}(t,0;s,z-a)\,Z(s,z+x)\,\ud z}
		=
		\frac{\bar{\mathcal{G}}(t,0;s,0)}{\int_{\R} p_{t-s}(z+a)\,\bar{\mathcal{G}}(t,0;s,z)\,Z(s,z+y)\,\ud z},
	\]
	where we changed variables $z\mapsto z-a$. Using again stationarity of
	$\{Z(s,z)\}_{z\in\R}$ in the spatial variable (and its independence from the
	noise on $(s,t]$), we may replace $Z(s,z+y)$ by $Z(s,z)$ in distribution,
	yielding \eqref{E:B-shift-shear}.

	The proof of \eqref{E:Btilde-shift} is the same, with $Z(s,\cdot)$ replaced by $1$, since
	$\widetilde Z(t,x;s)=\int_{\R} p_{t-s}(x-u)\,\bar{\mathcal{G}}(t,x;s,u)\,\ud u$.
	The shear-reduced form \eqref{E:Btilde-shift-shear} follows by the same stationarity argument as above.
\end{proof}

\begin{remark}
	Lemma~\ref{L:B-shift} is the flat-initial-data counterpart of
	\cite{gu.pu:25:spatial}*{Lemma~2.1}, which treats the narrow-wedge setting. In
	their case the denominator is the full Green's function
	$\bar{\mathcal{G}}(t,x;0,0)$ and the convolution kernel is a Brownian-bridge
	density $p_{s(t-s)/t}$; here, the denominator is the flat partition function
	$Z(t,x)$ (resp.\ $\widetilde Z(t,x;s)$) and the kernel is the heat kernel
	$p_{t-s}$. The proof strategy---translation stationarity followed by shear
	invariance---is the same in both cases. The identity is recorded to emphasize
	that both denominators $Z(t,x)$ and $\widetilde Z(t,x;s)$ lead to natural
	``one-point'' averages.
\end{remark}

\begin{lemma}[One-sided reduction to
		$\widetilde{\mathcal{B}}$]\label{L:B-tilde-mixed}
	Fix $t>0$ and $\varepsilon \in (0,2).$ As $|x| \to \infty$,
	\begin{equation} \label{E:c-tBB}
		\begin{aligned}
			\Cov[h(t,x),h(t,0)] =  \int_{0}^{|x|^{-2+\varepsilon}}\int_{\R}  \E\!\left[ \widetilde{\mathcal{B}}(t,x;s,y) \right] \, \E\!\left[\mathcal{B}(t,0;s,y)\right] \,
			K(t,s,x,y)\,\ud y\,\ud s \\
			+ o\!\left(|x|^{-2}\exp\!\left(-\frac{x^2}{4t}\right)\right).
		\end{aligned}
	\end{equation}
\end{lemma}
\begin{proof}
	By Lemma~\ref{L:B-s-r}, we have
	\[
		\begin{split}
			 & \int_{0}^{|x|^{-2+\varepsilon}}\int_{\mathbb{R}}
			\Big|
			\mathbb{E}\!\left[ \mathcal{B}(t,x;s,y) \right]
			-
			\mathbb{E}\!\left[ \widetilde{\mathcal{B}}(t,x;s,y) \right]
			\Big|
			\, \mathbb{E}\!\left[\mathcal{B}(t,0;s,y)\right]
			\, K(t,s,x,y)\, \ud y\,\ud s                        \\
			 & \lesssim
			\int_{0}^{|x|^{-2+\varepsilon}}\int_{\mathbb{R}}
			s^{1/4} \, K(t,s,x,y)\, \ud y\,\ud s
			\le
			|x|^{(-2+\varepsilon)/4}
			\int_{0}^{|x|^{-2+\varepsilon}}\int_{\mathbb{R}}
			K(t,s,x,y)\, \ud y\,\ud s.
		\end{split}
	\]
	This implies that \eqref{E:c-tBB} holds by the same arguments used in the
	proof of Lemma~\ref{L:Bcov0}.
\end{proof}

Similarly, by applying Lemma~\ref{L:B-s-r} also at $x=0$ and repeating the
argument in Lemma~\ref{L:B-tilde-mixed}, we can reduce both factors.

\begin{lemma}[Two-sided reduction to
		$\widetilde{\mathcal{B}}$]\label{L:B-tilde-both}
	Fix $t>0$ and $\varepsilon \in (0,2)$. As $|x| \to \infty$,
	\begin{equation}\label{E:c-tBB1}
		\begin{aligned}
			\Cov[h(t,x),h(t,0)]
			=
			\int_{0}^{|x|^{-2+\varepsilon}}\int_{\mathbb{R}}
			\mathbb{E}\!\left[ \widetilde{\mathcal{B}}(t,x;s,y) \right]\,
			\mathbb{E}\!\left[ \widetilde{\mathcal{B}}(t,0;s,y) \right]\,
			K(t,s,x,y)\, \ud y\,\ud s \\
			+
			o\!\left(|x|^{-2}\exp\!\left(-\frac{x^2}{4t}\right)\right).
		\end{aligned}
	\end{equation}
\end{lemma}
\begin{proof}
	Starting from Lemma~\ref{L:B-tilde-mixed} and using Lemma~\ref{L:B-s-r} at
	$x=0$, we have
	\[
		\begin{split}
			 & \int_{0}^{|x|^{-2+\varepsilon}}\int_{\mathbb{R}}
			\Big|
			\mathbb{E}\!\left[ \mathcal{B}(t,0;s,y) \right]
			-
			\mathbb{E}\!\left[ \widetilde{\mathcal{B}}(t,0;s,y) \right]
			\Big|
			\, \mathbb{E}\!\left[\widetilde{\mathcal{B}}(t,x;s,y)\right]
			\, K(t,s,x,y)\, \ud y\,\ud s                        \\
			 & \lesssim
			\int_{0}^{|x|^{-2+\varepsilon}}\int_{\mathbb{R}}
			s^{1/4} \, K(t,s,x,y)\,\ud y\,\ud s
			\le
			|x|^{(-2+\varepsilon)/4}
			\int_{0}^{|x|^{-2+\varepsilon}}\int_{\mathbb{R}}
			K(t,s,x,y)\,\ud y\,\ud s,
		\end{split}
	\]
	which is $o\!\left(|x|^{-2}\exp\!\left(-x^2/(4t)\right)\right)$ by the same
	argument as in Lemma~\ref{L:B-tilde-mixed} (using the boundary asymptotic for
	$I_{\mathrm{flat}}$). This yields \eqref{E:c-tBB1}.
\end{proof}
We will use Lemma~\ref{L:B-tilde-both} (i.e., \eqref{E:c-tBB1}) to reduce the
covariance asymptotics to the study of the deterministic kernel $K$ with the
future-only factors $\widetilde{\mathcal{B}}$.

\subsection{Boundary-layer reduction}\label{SS:boundary-layer-reduction}

The next result formalizes the boundary-layer heuristic that already appears in the
narrow-wedge analysis of Gu--Pu \cite{gu.pu:25:spatial}: the large-$|x|$ behavior of
$\Cov[h(t,x),h(t,0)]$ is controlled by times $s\lesssim |x|^{-2+\varepsilon}$ and by the
behavior of the boundary-layer factor $\E[\mathcal{B}(t,x;s,y)\mathcal{B}(t,0;s,y)]$
in the window $|y-x/2|=O(1)$.

\begin{theorem}[Reduction to a boundary-layer constant]\label{thm:c-redct}
	Fix $t>0$, and assume
	that for every $M>0$ there exists $\varepsilon\in(0,2)$ and a constant $C(t)$ such that
	\begin{equation}\label{E:key-a}
		\lim_{|x|\to\infty}\ \sup_{\substack{0<s\le |x|^{-2+\varepsilon}\\ |y-x/2|\le M}}
		\left|\E\!\left[\mathcal{B}(t,x;s,y)\,\mathcal{B}(t,0;s,y)\right]-C(t)\right|=0.
	\end{equation}
	Then
	\[
		\Cov[h(t,x),h(t,0)]
		\sim C(t)\,I_{\mathrm{flat}}(t,x),
		\qquad |x|\to\infty,
	\]
	that is, \eqref{conj:kpz-cov} holds with $\kappa(t)=C(t)$.
\end{theorem}
\begin{proof}[Proof of Theorem~\ref{thm:c-redct}]
	Fix $M>0$, and let $\varepsilon\in(0,2)$ and $C(t)$ satisfy \eqref{E:key-a}.
	Write
	$L_x\coloneqq |x|^{-2+\varepsilon}$, so that \eqref{E:r-BB} holds for this choice of $\varepsilon\in(0,2)$.
	On $(0,L_x)$, split the $y$-integral on \eqref{E:r-BB} into $\{|y-x/2|\le M\}$ and its
	complement.  By \eqref{E:key-a} and the nonnegativity of $K$,
	\[
		\int_0^{L_x}\int_{|y-x/2|\le M}
		\big(\E\!\left[\mathcal{B}(t,x;s,y)\,\mathcal{B}(t,0;s,y)\right]-C(t)\big)K(t,s,x,y)\ud y\,\ud s
		=o(1)\int_0^{L_x}\int_{\R}K(t,s,x,y)\ud y\,\ud s.
	\]
	For the tail $\{|y-x/2|>M\}$, use Cauchy--Schwarz together with \eqref{E:B-unif-L2} to get
	$|\E[\mathcal{B}(t,x;s,y)\,\mathcal{B}(t,0;s,y)]|\le C_t$, and hence
	\[
		\left|\int_0^{L_x}\int_{|y-x/2| > M}
		\big(\E\!\left[\mathcal{B}(t,x;s,y)\,\mathcal{B}(t,0;s,y)\right]-C(t)\big)K(t,s,x,y)\ud y\,\ud s \right| \le C\int_0^{L_x}\int_{|y-x/2|> M}K(t,s,x,y)\ud y\,\ud s.
	\]
	Finally, note that for $a\coloneqq t-s$ one has the exact identity
	\[
		K(t,s,x,y)=p_a(x-y)p_a(y)=p_{2a}(x)\,p_{a/2}\!\left(y-\frac{x}{2}\right).
	\]
	In particular,
	\[
		\int_{|y-x/2|>M}K(t,s,x,y)\,\ud y
		=p_{2a}(x)\int_{|u|>M}p_{a/2}(u)\,\ud u
		\le p_{2a}(x)\int_{|u|>M}p_{t/2}(u)\,\ud u,
	\]
	while $\int_{\R}K(t,s,x,y)\,\ud y=p_{2a}(x)$.
	Hence, by choosing $M$ large enough that $\int_{|u|>M}p_{t/2}(u)\,\ud u\le \eta$,
	the right-hand side in the previous display can be made $\le \eta$ times
	$\int_0^{L_x}\int_{\R}K(t,s,x,y)\,\ud y\,\ud s$ (uniformly in $x$).
	Since $\eta>0$ is arbitrary, the tail contribution is $o(1)$ times
	$\int_0^{L_x}\int_{\R}K(t,s,x,y)\,\ud y\,\ud s$. Combining this with the local estimate from \eqref{E:key-a} yields
	\[
		\Cov[h(t,x),h(t,0)]
		=
		C(t)\int_0^{L_x}\int_{\R}K(t,s,x,y)\,\ud y\,\ud s
		+o(1)\int_0^{L_x}\int_{\R}K(t,s,x,y)\,\ud y\,\ud s.
	\]
	Finally,
	\[
		\int_0^{L_x}\int_{\R}K(t,s,x,y)\,\ud y\,\ud s
		=
		I_{\mathrm{flat}}(t,x)-\int_{L_x}^t p_{2(t-s)}(x)\,\ud s.
	\]
	For $s\in[L_x,t)$ one has $t-s\le t-L_x$, hence
	$p_{2(t-s)}(x)\le C_t\exp\!\big(-x^2/(4(t-L_x))\big)
		=C_t\exp\!\big(-x^2/(4t)\big)\exp\!\big(-c_t x^2 L_x\big)$.
	Since $x^2 L_x=|x|^{\varepsilon}\to\infty$, the tail integral is
	$o\!\big(|x|^{-2}\exp(-x^2/(4t))\big)$, and therefore
	$\int_0^{L_x}\int_{\R}K(t,s,x,y)\,\ud y\,\ud s=I_{\mathrm{flat}}(t,x)\big(1+o(1)\big)$
	by \eqref{E:Iflat-asymp}.
	This completes the proof.
\end{proof}

To complete the reduction in Theorem~\ref{thm:c-redct} and hence prove
Theorem~\ref{thm:flat-main}, it remains to establish the boundary-layer estimate
\eqref{E:key-a}. We prove it via the future-only profile
$\widetilde{\mathcal{B}}$ and shear mixing.

\subsection{Boundary-layer estimate via the future-only profile}\label{SS:future-only}

\paragraph{Future-only profile and boundary-layer constant.} Recall the
future-only partition function $\widetilde Z(t,x;s)$ from~\eqref{E:Ztilde} and
the deterministic factor $\widetilde{\mathcal{B}}(t,x;s,y)$
from~\eqref{E:Btilde}. By translation invariance of the noise, it depends only
on the time increment $t-s$ and the spatial displacement $x-y$: there exists a
deterministic function $\widetilde b_{\tau}(\cdot)$ such that
\begin{equation}\label{E:Btilde-profile}
	\widetilde{\mathcal{B}}(t,x;s,y)=\widetilde b_{t-s}(x-y),\qquad 0<s<t,\ x,y\in\R.
\end{equation}
(One may verify \eqref{E:Btilde-profile} by shifting space by $-y$ and time by
$-s$ in the definition above.)

The two-sided reduction Lemma~\ref{L:B-tilde-both} therefore yields, in the
boundary-layer regime,
\begin{equation}\label{E:cov-Btilde-profile}
	\Cov[h(t,x),h(t,0)]
	=
	\int_{0}^{|x|^{-2+\varepsilon}}\!\!\int_{\R}
	\widetilde b_{t-s}(x-y)\,\widetilde b_{t-s}(-y)\,
	K(t,s,x,y)\,\ud y\,\ud s
	\;+\;
	o\!\left(|x|^{-2}e^{ -\frac{x^2}{4t} }\right).
\end{equation}
In view of the identity $K(t,s,x,y)=p_{2(t-s)}(x)\,p_{(t-s)/2}(y-x/2)$, the
$y$-integral in \eqref{E:cov-Btilde-profile} is concentrated on $|y-x/2|\lesssim
	1$.

\medskip

The following is the natural ``future-only'' target, analogous in spirit to the
key technical estimate of Gu--Pu (their
Proposition~\cite{gu.pu:25:spatial}*{Proposition~3.1}), but adapted to the flat
initial profile.

\begin{proposition}[Flat covariance constant from the future-only
		profile]\label{P:flat-constant-from-Btilde}
	Fix $t>0$. Assume there exists a constant $\kappa_\ast(t)\in(0,\infty)$ such
	that for every $M>0$ there exists $\varepsilon\in(0,2)$ for which
	\begin{equation}\label{E:Btilde-boundary-layer-const}
		\lim_{|x|\to\infty}\ \sup_{\substack{0<s\le |x|^{-2+\varepsilon}\\ |y-x/2|\le M}}
		\Big(
		\left|\widetilde{\mathcal{B}}(t,x;s,y)-\kappa_\ast(t)\right|
		+
		\left|\widetilde{\mathcal{B}}(t,0;s,y)-\kappa_\ast(t)\right|
		\Big)=0.
	\end{equation}
	Then the flat covariance asymptotic \eqref{E:flat-goal} holds with
	$\kappa(t)=\kappa_\ast(t)^2$.
\end{proposition}

\begin{proof}
	Under~\eqref{E:Btilde-boundary-layer-const}, one has $\widetilde
		b_{t-s}(x-y)\widetilde b_{t-s}(-y)=\kappa_\ast(t)^2+o(1)$ uniformly on the
	boundary layer, since $t-s\to t$ uniformly over $0<s\le |x|^{-2+\varepsilon}$.
	Plugging this into \eqref{E:cov-Btilde-profile} and using
	$\int_{\R}K(t,s,x,y)\,\ud y=p_{2(t-s)}(x)$ yields
	\[
		\Cov[h(t,x),h(t,0)]=\kappa_\ast(t)^2\,I_{\mathrm{flat}}(t,x)\big(1+o(1)\big),
	\]
	hence \eqref{E:flat-goal}.
\end{proof}

\begin{proposition}[Boundary-layer constant for
		$\widetilde{\mathcal{B}}$]\label{P:Btilde-const}
	Fix $t>0$, and set
	\[
		\kappa_\ast(t)\coloneqq \E\!\left[Z(t,0)^{-1}\right]\in(1,\infty).
	\]
	Then for every $M>0$ and every $\varepsilon\in(0,2)$,
	\eqref{E:Btilde-boundary-layer-const} holds with this choice of
	$\kappa_\ast(t)$.
\end{proposition}
\begin{proof}
	Fix $M>0$ and $\varepsilon\in(0,2)$. We first show that
	$\widetilde{\mathcal{B}}(t,x;s,y)\to \kappa_\ast(t)$ uniformly on the
	boundary layer. Let $0<s<t$ and $x,y\in\R$, and set $a\coloneqq y-x$. Since
	$\widetilde{\mathcal{B}}(t,x;s,y)$ is deterministic,
	\eqref{E:Btilde-shift-shear} yields
	\begin{equation}\label{E:Btilde-ratio}
		\widetilde{\mathcal{B}}(t,x;s,y)
		=
		\E\!\left[
			\frac{\bar{\mathcal{G}}(t,0;s,0)}{\widetilde D_{t,s}(a)}
			\right],
	\end{equation}
	where
	\begin{equation}\label{E:Dtilde-def}
		\widetilde D_{t,s}(a)
		\coloneqq
		\int_{\R} p_{t-s}(z+a)\,\bar{\mathcal{G}}(t,0;s,z)\,\ud z
		=
		\int_{\R} p_{t-s}(u)\,\bar{\mathcal{G}}(t,0;s,u-a)\,\ud u.
	\end{equation}
	Note that $\widetilde D_{t,s}(0)=\widetilde Z(t,0;s)$, and hence $\widetilde
		D_{t,s}(0)$ has the same law as $Z(t-s,0)$ by stationarity of the noise. In
	particular,
	\[
		\E\!\left[\widetilde D_{t,s}(0)^{-1}\right]=\E\!\left[Z(t-s,0)^{-1}\right].
	\]

	By the shear invariance of the normalized Green's
	function~\cite{alberts.janjigian.ea:22:greens}*{Proposition~2.3(iii)} (applied
	with $r=t$), for $\nu\coloneqq a/(t-s)$ one has $\bar{\mathcal{G}}(t,0;s,u-a)
		= \bar{\mathcal{G}}(t,0;s,u)\circ \theta_{t,\nu}$, and therefore, by
	\eqref{E:Dtilde-def},
	\begin{equation}\label{E:Dtilde-shear}
		\widetilde D_{t,s}(a)=\widetilde D_{t,s}(0)\circ \theta_{t,\nu},
		\qquad
		\widetilde D_{t,s}(a)^{-1}=\widetilde D_{t,s}(0)^{-1}\circ \theta_{t,\nu}.
	\end{equation}

	Using \eqref{E:Btilde-ratio} and \eqref{E:Dtilde-shear}, we can rewrite
	\[
		\widetilde{\mathcal{B}}(t,x;s,y)-\E\!\left[Z(t-s,0)^{-1}\right]
		=
		\Cov\!\left(\bar{\mathcal{G}}(t,0;s,0),\ \widetilde D_{t,s}(0)^{-1}\circ \theta_{t,\nu}\right).
	\]
	By Lemma~\ref{L:barG-moments} and Proposition~\ref{P:Z-moments} applied to
	$\widetilde Z(t,0;s)$ (which has the same law as $Z(t-s,0)$), one has
	$\bar{\mathcal{G}}(t,0;s,0)\in L^2(\Omega)$ and $\widetilde D_{t,s}(0)^{-1}\in
		L^2(\Omega)$. Hence Lemma~\ref{L:shear-mixing} (see
	Section~\ref{SS:shear-mixing}) implies that for each fixed $s\in(0,t)$,
	\begin{equation}\label{E:Btilde-fixed-s}
		\lim_{|a|\to\infty}
		\left|\widetilde{\mathcal{B}}(t,x;s,y)-\E\!\left[Z(t-s,0)^{-1}\right]\right|=0, \qquad a=y-x.
	\end{equation}

	\medskip

	We now upgrade \eqref{E:Btilde-fixed-s} to the uniform boundary-layer limit
	\eqref{E:Btilde-boundary-layer-const}. Fix $M>0$ and consider $0<s\le
		|x|^{-2+\varepsilon}$ and $|y-x/2|\le M$. In this regime, $s\to0$ and
	$|a|=|y-x|\to\infty$ as $|x|\to\infty$ (uniformly over $|y-x/2|\le M$), and
	moreover $t-s\ge t/2$ for $|x|\gg 1$.

	Write $F_s \coloneqq \bar{\mathcal{G}}(t,0;s,0)$ and $G_s\coloneqq \widetilde
		D_{t,s}(0)^{-1}$.
	By~\cite{alberts.janjigian.ea:22:greens}*{Proposition~3.8(ii)} (in particular,
	the H\"older continuity in the initial time variable for the normalized
	Green's function), one has $F_s\to F_0$ in $L^2(\Omega)$ as $s\downarrow 0$.
	Similarly, by \eqref{E:Z-vs-Ztilde-smalltime} and
	Proposition~\ref{P:Z-moments}, $G_s\to G_0$ in $L^2(\Omega)$ as $s\downarrow
		0$.

	Let $s\in(0,t/2]$ and $\nu=a/(t-s)$. Using that $\theta_{t,\nu}$ is measure
	preserving, we have
	\[
		\left|\Cov(F_s,\,G_s\circ \theta_{t,\nu})-\Cov(F_0,\,G_0\circ \theta_{t,\nu})\right|
		\le \Norm{F_s-F_0}_2\,\Norm{G_s}_2 + \Norm{F_0}_2\,\Norm{G_s-G_0}_2,
	\]
	which tends to $0$ uniformly over $|a|\ge 1$ as $s\downarrow 0$ by the uniform
	$L^2$ bounds on $F_s$ and $G_s$.

	Since $|a|/(t-s)\to\infty$ as $|x|\to\infty$ in the boundary-layer geometry,
	Lemma~\ref{L:shear-mixing} yields $\Cov(F_0, G_0\circ \theta_{t,\nu})\to0$ as
	$|\nu|\to\infty$, and therefore the covariance term in
	$\widetilde{\mathcal{B}}(t,x;s,y)-\E[Z(t-s,0)^{-1}]$ vanishes uniformly on the
	boundary layer. Finally, by Lemma~\ref{L:Z-inv-cont}, $\E[Z(t-s,0)^{-1}]\to
		\E[Z(t,0)^{-1}]=\kappa_\ast(t)$ as $s\downarrow 0$.

	\medskip

	It remains to control $\widetilde{\mathcal{B}}(t,0;s,y)$ on the same boundary
	layer. Set $a_0\coloneqq y$ and observe that $|a_0|\to\infty$ as $|x|\to\infty$
	uniformly over $|y-x/2|\le M$, since $|y|\ge |x|/2-M$. Repeating the argument
	above with $x=0$ (so that $a=a_0$ and $\nu=a_0/(t-s)$) yields
	\[
		\sup_{\substack{0<s\le |x|^{-2+\varepsilon}\\ |y-x/2|\le M}}
		\left|\widetilde{\mathcal{B}}(t,0;s,y)-\kappa_\ast(t)\right|\longrightarrow 0,
		\qquad |x|\to\infty.
	\]
	Together with the first part (for $\widetilde{\mathcal{B}}(t,x;s,y)$), this
	proves \eqref{E:Btilde-boundary-layer-const}.
\end{proof}

\begin{corollary}[Boundary-layer estimate \eqref{E:key-a}]\label{C:key-a}
	Fix $t>0$, and set
	\[
		C(t)\coloneqq \kappa_\ast(t)^2=\Big(\E\!\left[Z(t,0)^{-1}\right]\Big)^2.
	\]
	Then for every $M>0$ there exists $\varepsilon\in(0,2)$ such
	that~\eqref{E:key-a} holds with this choice of $C(t)$.
\end{corollary}
\begin{proof}
	By Proposition~\ref{P:Btilde-const}, \eqref{E:Btilde-boundary-layer-const}
	holds with $\kappa_\ast(t)=\E[Z(t,0)^{-1}]$. Therefore,
	\[
		\widetilde{\mathcal{B}}(t,x;s,y)\,\widetilde{\mathcal{B}}(t,0;s,y)
		\longrightarrow
		\kappa_\ast(t)^2=C(t)
	\]
	uniformly on the boundary layer (with the same $M$ and some
	$\varepsilon\in(0,2)$). Combining this with Lemma~\ref{L:key-a-to-Btilde}
	(which bounds the difference between $\E[\mathcal{B}\mathcal{B}]$ and
	$\widetilde{\mathcal{B}}\,\widetilde{\mathcal{B}}$ by $C_t s^{1/4}$) yields
	\eqref{E:key-a}.
\end{proof}

\subsection{Proof of Theorem~\ref{thm:flat-main}}\label{SS:proof-main}

\begin{proof}[Proof of Theorem~\ref{thm:flat-main}]
	By Corollary~\ref{C:key-a}, the boundary-layer estimate \eqref{E:key-a} holds
	with $C(t)=\big(\E[Z(t,0)^{-1}]\big)^2$. Therefore Theorem~\ref{thm:c-redct}
	applies and yields
	\[
		\Cov[h(t,x),h(t,0)]\sim \Big(\E\!\left[Z(t,0)^{-1}\right]\Big)^2 I_{\mathrm{flat}}(t,x),
		\qquad |x|\to\infty.
	\]
	Since $I_{\mathrm{flat}}(t,x)=\int_0^t p_{2r}(x)\,\ud r$
	by~\eqref{E:Iflat-def}, this is exactly the statement of
	Theorem~\ref{thm:flat-main}.
\end{proof}

\section*{Acknowledgments}
The authors thank Yu Gu and Fei Pu for helpful comments on an earlier version
of this manuscript.
L.~C. was partially supported by NSF grants DMS-2246850/2443823 and a
Collaboration Grant for Mathematicians (\#959981) from the Simons Foundation. J.~J.\ was also partially supported by NSF grants DMS-2246850/2443823.


\end{document}